\documentclass[reqno, british]{amsart}

\usepackage{geometry}
\usepackage{mathrsfs}
\usepackage{enumitem}
\usepackage{amsthm}
\usepackage{amsmath}
\usepackage{amssymb}
\usepackage{hyperref}
\usepackage{graphicx, xypic}
\usepackage{float}

\makeatletter
\theoremstyle{plain}
\newtheorem{thm}{\protect\theoremname}[section]
\theoremstyle{plain}

\theoremstyle{definition}
\newtheorem{defn}[thm]{\protect\definitionname}
\theoremstyle{plain}
\newtheorem{lem}[thm]{\protect\lemmaname}
\theoremstyle{plain}

\theoremstyle{definition}
\newtheorem{example}[thm]{\protect\examplename}
\theoremstyle{remark}
\newtheorem{rem}[thm]{\protect\remarkname}
\theoremstyle{plain}
\newtheorem{cor}[thm]{\protect\corollaryname}
\newtheorem*{cor*}{Corollary}
\theoremstyle{plain}

 \newlist{casenv}{enumerate}{4}
 \setlist[casenv]{leftmargin=*,align=left,widest={iiii}}
 \setlist[casenv,1]{label={{\itshape\ \casename} \arabic*.},ref=\arabic*}
 \setlist[casenv,2]{label={{\itshape\ \casename} \roman*.},ref=\roman*}
 \setlist[casenv,3]{label={{\itshape\ \casename\ \alph*.}},ref=\alph*}
 \setlist[casenv,4]{label={{\itshape\ \casename} \arabic*.},ref=\arabic*}

\usepackage{bbold}
\DeclareSymbolFont{bbold}{U}{bbold}{m}{n}
\DeclareSymbolFontAlphabet{\mathbbold}{bbold}
\makeatother

\usepackage{babel}
\providecommand{\definitionname}{Definition}
\providecommand{\examplename}{Example}
\providecommand{\lemmaname}{Lemma}
\providecommand{\remarkname}{Remark}
\providecommand{\theoremname}{Theorem}
\providecommand{\claimname}{Claim}
\providecommand{\casename}{Case}
\providecommand{\corollaryname}{Corollary}
\providecommand{\factname}{Fact}
\providecommand{\propositionname}{Proposition}

\newcommand\Cref[1]{{Corollary~\ref{#1}}}
\newcommand\Lref[1]{{Lemma~\ref{#1}}}
\newcommand\Tref[1]{{Theorem~\ref{#1}}}

\newcommand\Sref[1]{{section~\ref{#1}}}
\newcommand\sSref[1]{{subsection~\ref{#1}}}

\newcommand\Dref[1]{{Definition~\ref{#1}}}
\newcommand\Rref[1]{{Remark~\ref{#1}}}
\newcommand\Eref[1]{{Example~\ref{#1}}}

\global\long\def\R{\mathcal{R}}
\global\long\def\Span{\textnormal{Span}}
\global\long\def\gl{\textnormal{gl}}
\global\long\def\tr{\textnormal{tr}}

\global\long\def\ad{\textnormal{ad}}

\global\long\def\layer#1#2{\overset{\left[#2\right]}{}#1}
\global\long\def\zero{1_{R}^\circ}

\global\long\def\zeroM{0_M}
\global\long\def\one{1_{R}}
\global\long\def\minus{\left(-\right)}
\global\long\def\zeroset#1{{#1}^{\circ}}
\global\long\def\nzeroset#1{{#1}^{\vee}}

\global\long\def\sl{\textnormal{sl}}

\global\long\def\rk{\textnormal{rk}}
\global\long\def\ideal{\vartriangleleft}

\global\long\def\sign{\textnormal{sign}}

\global\long\def\Rad{\textnormal{Rad}}
\global\long\def\SRad{\textnormal{SRad}}
\global\long\def\ELT#1#2{\mathscr{R}\left(#2,#1\right)}
\global\long\def\minf{0_{R}}
\global\long\def\ch{\textnormal{char}}
\global\long\def\Ad{\textnormal{Ad}}
\global\long\def\L{\mathscr{L}}
\global\long\def\F{\mathscr{F}}

\global\long\def\t{\tau}

\global\long\def\etr{\textnormal{etr}}

\global\long\def\Hom#1#2{\textnormal{Hom}\left(#1,#2\right)}
\global\long\def\End#1{\textnormal{End}\left(#1\right)}
\global\long\def\Im{\textnormal{Im}}
\global\long\def\kersim{\mathcal{M}}
\global\long\def\quo#1#2{\raisebox{.2em}{\ensuremath{#1}}\left/\raisebox{-.2em}{\ensuremath{#2}}\right.}

\global\long\def\kil#1#2{\kappa\left(#1,#2\right)}
\global\long\def\esskil#1#2{\kappa^{\mathrm{es}}\left(#1,#2\right)}
\global\long\def\T{\mathcal{T}}
\global\long\def\G{\mathcal{G}}
\global\long\def\Gz{\mathcal{G}_{\mathbbold{0}}}

\global\long\def\symmdash{\succeq}
\global\long\def\asymmdash{\preceq}
\global\long\def\symmeq{\equiv_{\circ}}
\global\long\def\N{\mathbb{N}}
\global\long\def\Z{\mathbb{Z}}

\global\long\def\zerog{\mathbbold{0}}
\global\long\def\puis#1{#1\left\{\left\{t\right\}\right\}}

\usepackage{bbold}
\DeclareSymbolFont{bbold}{U}{bbold}{m}{n}
\DeclareSymbolFontAlphabet{\mathbbold}{bbold}

\title{Modules and Lie Semialgebras over Semirings with a Negation Map}
\author[Guy Blachar]{Guy Blachar}
\address{Department of Mathematics, Bar-Ilan University, Ramat-Gan 52900,
Israel.} \email{\href{mailto:blachag@biu.ac.il}{blachag@biu.ac.il}}
\date{\today}

\thanks{This article contains work from the author's M.Sc.\ Thesis, which was submitted to the Math Department at Bar-Ilan University. The work was carried under the supervision of Prof.\ Louis Rowen from Bar-Ilan University, to whom the author thanks deeply for his help and guidance.}

\begin{document}

\maketitle
\begin{abstract}
In this article, we present the basic definitions of modules and Lie semialgebras over semirings with a negation map. Our main example of a semiring with a negation map is ELT algebras, and some of the results in this article are formulated and proved only in the ELT theory.

When dealing with modules, we focus on linearly independent sets and spanning sets. We define a notion of lifting a module with a negation map, similarly to the tropicalization process, and use it to prove several theorems about semirings with a negation map which possess a lift.

In the context of Lie semialgebras over semirings with a negation map, we first give basic definitions, and provide parallel constructions to the classical Lie algebras. We prove an ELT version of Cartan's criterion for semisimplicity, and provide a counterexample for the naive version of the PBW Theorem.
\end{abstract}

\tableofcontents
\addtocontents{toc}{~\hfill\textbf{Page}\par}
\section{Introduction}

This paper's objective is to form an algebraic basis in the context of semirings with negation maps.\\

Semirings need not have additive inverses to all of the elements. While some of the theory of rings can be copied ``as-is'' to semirings, there are many facts about rings which use the additive inverses of the elements. The idea of negation maps on semirings is to imitate the additive inverse map. In order to do that, we assume that we are given a map $a\mapsto\minus a$, which satisfies all of the properties of the usual inverse, except for the fact that $a+\minus a=0$. This allows us to use the concept of additive inverses, even if they do not exist in the usual definition. Semirings with a negation map are discussed in \cite{Akian1990, Gaubert1992, Gaubert1997, Akian2008, Akian2014, Rowen2016}.\\

During this paper, we will first deal with modules which have a negation map over a semiring, and then study Lie semialgebras with a negation map over semirings.\\

In the context of modules, our main interest would be to study linearly independent sets and spanning sets. We will study maximal linearly independent sets, minimal spanning sets and the connection between them.\\

As the introduction demonstrates, the tropical world provides us with several nontrivial examples of semirings with a negation map -- supertropical algebras and Exploded Layered Tropical (ELT) algebras. In the study of tropical algebra and tropical geometry, one tool is using Puiseux series -- which allow us to ``lift'' a tropical problem to the classical algebra, and using the classical tools to solve it. We will give a general definition of what a lift is, and study some of its properties.\\

We will move to considering Lie semialgebras with a negation map. We will study at first the basic definitions related to Lie algebras when working over semirings. We will consider on Lie semialgebras which are free as modules, such that their basis contains 1, 2 or 3 elements, and also consider nilpotent, solvable and semisimple Lie semialgebras.\\

Cartan's criterion regarding semisimple Lie algebras states that a Lie algebra is semisimple if and only if its Killing form is nondegenerate. In this paper, we prove an ELT version of this theorem for Lie semialgebras over ELT algebras.\\

Another point of interest is Poincar\'{e}-–Birkhoff-–Witt (PBW) Theorem. In the classical theory of Lie algebras, the PBW Theorem states that every Lie algebra can be embedded into an associative algebra (its universal enveloping algebra), where the Lie bracket on this algebra is the commutator. We will construct the universal enveloping algebra of a Lie semialgebra, and give a counterexample to the naive version of the PBW theorem.\\

Throughout this paper, addition and multiplication will always be denoted as $+$ and $\cdot$, in order to emphasize the algebraic structure. Negation maps will always be denoted by $\minus$. $\N$ stands for the set of natural numbers, whereas $\N_0=\N\cup\left\{0\right\}$. $\Z$ stands for the ring of integers, and $\quo{\Z}{n\Z}$ denotes the quotient ring of $\Z$ by $n\Z$.

\subsection{Semirings with a Negation Map}
\begin{defn}
Let $R$ be a semiring. A map $\minus:R\to R$ is a \textbf{negation map} (or a \textbf{symmetry}) on $R$ if the following properties hold:
\begin{enumerate}
\item $\forall a,b\in R:\minus\left(a+b\right)=\minus a+\minus b$.
\item $\minus0_R=0_R$.
\item $\forall a,b\in R:\minus\left(a\cdot b\right)=a\cdot\left(\minus b\right)=\left(\minus a\right)\cdot b$.
\item $\forall a\in R:\minus\left(\minus a\right)=a$.
\end{enumerate}
We say that $\left(R,\minus\right)$ is a \textbf{semiring with a negation map}. If $\minus$ is clear from the context, we will not mention it.
\end{defn}

For example, the map $\minus a=a$ defines a trivial negation map on every semiring $R$. If $R$ is a ring, it has a negation map given by $\minus a=-a$.\\

Throughout this article, we use the following notations:
\begin{itemize}
\item $a\minus a$ is denoted $a^\circ$.
\item $\zeroset{R}=\left\{a^\circ\middle|a\in R\right\}$.
\item $\nzeroset{R}=R\setminus\zeroset{R}\cup\left\{\minf\right\}$.
\item We define two partial orders on $R$:
\begin{itemize}
\item The relation $\symmdash$ (which we call the \textbf{surpassing} relation) defined by
$$a\symmdash b\Leftrightarrow \exists c\in R^\circ:a=b+c$$
\item The relation $\nabla$ defined by
$$a\nabla b\Leftrightarrow a\minus b\in R^\circ$$
\end{itemize}
\end{itemize}

\subsection{Modules Over Semirings with a Negation Map}

We now consider modules over semirings with a negation map.

\begin{defn}
Let $R$ be a semiring with a negation map. A \textbf{(left) $R$-module} is a commutative monoid $\left(M,+,\zeroM\right)$, equipped with an operation of ``scalar multiplication'', $ R\times M\rightarrow M$, where we denote $\left(\alpha,x\right)\mapsto\alpha x$, such that the following properties hold:
\begin{enumerate}
\item $\forall\alpha\in R\,\forall x,y\in M:\alpha\left(x+y\right)=\alpha x+\alpha y$.
\item $\forall\alpha,\beta\in R\,\forall x\in M:\left(\alpha+\beta\right)x=\alpha x+\beta x$.
\item $\forall\alpha,\beta\in R\,\forall x\in M:\alpha\left(\beta x\right)=\left(\alpha\cdot\beta\right)x$.
\item $\forall\alpha\in R:\alpha\zeroM=\zeroM$.
\item $\forall x\in M:\minf x=\zeroM$.
\item $\forall x\in M:\one x=x$.
\end{enumerate}

A right $R$-module is defined similarly.
\end{defn}

\begin{example}\label{exa:examples-of-modules}
The following are examples of $R$-modules, where $R$ is a semiring:
\begin{enumerate}
\item $R$ is an $R$-module, where the scalar multiplication is the multiplication of $R$. The submodules of $R$ are its \textbf{left ideals}, meaning sets $I\subseteq R$ that are closed under addition and satisfy $RI\subseteq I$.
\item Taking the direct sum of $M_i=R$ for $i\in I$, we obtain the \textbf{free $R$-module}:
    $$R^I=\bigoplus_{i\in I}R$$
    Two important special cases are:
\item $M_{m\times n}\left(R\right)$ is an $R$-module with the standard scalar multiplication, $\left(\alpha A\right)_{ij}=\alpha\cdot\left(A\right)_{ij}$.
\item Let $\left\{M_i\right\}_{i\in I}$ be $R$-modules. Then $\displaystyle{\prod_{i\in I}M_i}$ and $\displaystyle{\bigoplus_{i\in I}M_i}$ are $R$-modules with the usual componentwise sum and scalar multiplication.

\end{enumerate}
\end{example}

As in \cite{Rowen2016}, we consider negation maps on modules.

\begin{defn}
Let $R$ be a semiring, and let $M$ be an $R$-module. A map $\minus:M\to M$ is a \textbf{negation map} (or a \textbf{symmetry}) on $M$ if the following properties hold:
\begin{enumerate}
  \item $\forall x,y\in M:\minus\left(x+y\right)=\minus x+\minus y$.
  \item $\minus 0_M=0_M$.
  \item $\forall\alpha\in R\,\forall x\in M:\minus\left(\alpha x\right)=\alpha\left(\minus x\right)$.
  \item $\forall x\in M:\minus\left(\minus x\right)=x$.
\end{enumerate}
\end{defn}

If the underlying semiring has a negation map $\minus$, every $R$-module $M$ has an induced negation map given by $\minus x=\left(\minus 1_R\right)x$. Unless otherwise written, when working with a module over a semiring with a negation map, the negation map will be the induced one. We note that if it is not mentioned that the semiring has a negation map, the negation map on the module can be arbitrary (although the main interest is with the induced negation map).

\begin{defn}
Let $R$ be a semiring, and let $M$ be an $R$-module with a negation map. A \textbf{submodule with a negation map} of $M$ is a submodule $N$ of $M$ which is closed under the negation map. Since we are in the context of negation maps, and every module in this paper has a negation map, we will write ``submodule'' for a ``submodule with a negation map'', where we understand that every submodule should be closed under the negation map.
\end{defn}

A main example of semirings and modules with a negation map uses the process of symmetrization defined in \cite{Rowen2016}. Briefly, If $R$ is an arbitrary semiring, we may define $\hat{R}=R\times R$ with componentwise addition, and with multiplication given by
$$\left(r_1,r_2\right)\cdot\left(r'_1,r'_2\right)=\left(r_1r'_1+r_2r'_2,r_1r'_2+r_2r'_1\right).$$
This is indeed a semiring with a negation map given by $\left(r_1,r_2\right)\mapsto\left(r_2,r_1\right)$. Of course, there is a natural injection $R\hookrightarrow\hat{R}$ by $r\mapsto\left(r,0\right)$.\\

The idea of this construction is to imitate the way $\mathbb{Z}$ is constructed from $\mathbb{N}$. One should think of the element $\left(r_1,r_2\right)$ as $r_1+\minus r_2$.\\

Now, if $M$ is any $R$-module, we may define $\hat{M}=M\times M$. We want to define it as an $\hat{R}$-module; so we use componentwise addition, and define multiplication by
$$\left(r_1,r_2\right)\left(x_1,x_2\right)=\left(r_1x_1+r_2x_2,r_1x_2+r_2x_1\right).$$
This multiplication is called the \textbf{twist action} on $\hat{M}$ over $\hat{R}$. It endows $\hat{M}$ with a $\hat{R}$-module structure, and the induced negation map is
$$\minus\left(x_1,x_2\right)=\left(0,1\right)\left(x_1,x_2\right)=\left(x_2,x_1\right)$$

For the remainder of the introduction, we will present our two main examples of semirings with a negation map -- supertropical algebras and ELT algebras.

\subsection{Supertropical Algebras}

Supertropical algebras are a refinement of the usual max-plus algebra. This is a refinement of the max-plus algebra, in which one adds ``ghost elements'' to replace the role of the classical zero. Supertropical algebras are discussed in several articles, including \cite{Izhaki2008, Izhaki2008b, Izhaki2009, Izhaki2010, Izhaki2010b, Izhakian2015}.

\begin{defn}
A \textbf{supertropical semiring} is a quadruple $R:=\left(R,\T,\G,\nu\right)$, where $R$ is a semiring, $\T\subseteq R$ is a multiplicative submonoid, and $\Gz=\G\cup\left\{\zerog\right\}\subseteq R$ is an ordered semiring ideal, together with a map $\nu:R\to\Gz$, satisfying $\nu^2=\nu$ as well as the conditions:
$$a+b=\left\{\begin{matrix}a,&\nu\left(a\right)>\nu\left(b\right)\\\nu\left(a\right)&\nu\left(a\right)=\nu\left(b\right)\end{matrix}\right.$$
\end{defn}

The monoid $\T$ is called the monoid of \textbf{tangible elements}, while the elements of $\G$ are called \textbf{ghost elements}, and $\nu:R\to\Gz$ is called the \textbf{ghost map}. Intuitively, the tangible elements correspond to the original max-plus algebra, although now $a+a=\nu\left(a\right)$ instead of $a+a=a$.\\

Any supertropical semiring $R$ is a semiring with a negation map, when the negation map is~$\minus a=a$. Endowed with this negation map, $\zeroset{R}=\Gz$.

\subsection{Exploded Layered Tropical Algebras}

ELT algebras are a more refined degeneration of the classical algebra than the classical max-plus algebra. The main idea which stands behind this structure is not only to remember the element, but rather remember also another information -- its layer -- which tells us ``how many times we added this element to itself''. ELT algebras originate in \cite{Parker2012}, were formally defined in \cite{Sheiner2015} and are discussed in \cite{Blachar2016, Blachar2016b}.

\begin{defn}
Let $\L$ be a semiring, and $\F$ a totally ordered semigroup. An \textbf{Exploded Layered Tropical algebra} (or, in short, an \textbf{ELT algebra}) is the pair $\R=\ELT{\F}{\L}$, whose elements are denoted $\layer a{\ell}$ for $a\in\F$ and $\ell\in\L$, together with the semiring (without zero) structure:
\begin{enumerate}
\item $\layer{a_{1}}{\ell_{1}}+\layer{a_{2}}{\ell_{2}}:=\begin{cases}
\layer{a_{1}}{\ell_{1}} & a_{1}>a_{2}\\
\layer{a_{2}}{\ell_{2}} & a_{1}<a_{2}\\
\layer{a_{1}}{\ell_{1}+_\L\ell_{2}} & a_{1}=a_{2}
\end{cases}$.
\item $\layer{a_{1}}{\ell_{1}}\cdot\layer{a_{2}}{\ell_{2}}:=\layer{\left(a_{1}+_\F a_{2}\right)}{\ell_{1}\cdot_\L\ell_{2}}$.
\end{enumerate}
We write . For $\layer{a}{\ell}$, $\ell$ is called the \textbf{layer}, whereas $a$ is called the \textbf{tangible value}.
\end{defn}

Let $\R$ be an ELT algebra. We write $s:\R\rightarrow\L$ for the projection on the first component (the \textbf{sorting map}):
$$s\left(\layer a{\ell}\right)=\ell$$
We also write $\t:\R\rightarrow\F$ for the projection on the second component:
$$\t\left(\layer a{\ell}\right)=a$$
We denote the \textbf{zero-layer subset}
$$\zeroset{\R}=\left\{\alpha\in \R\middle| s\left(\alpha\right)=0\right\}$$
and
$$\R^*=\left\{\alpha\in\R\middle|s\left(\alpha\right)\neq 0\right\}=\R\setminus\zeroset{\R}$$

\begin{defn}
An ELT algebra $\R=\ELT{\L}{\F}$ in which $\F$ is a totally ordered group and $\L$ is a ring (with $1$) is called an \textbf{ELT ring}.
\end{defn}

\begin{defn}
For any ELT ring $\R$, we define $\minus\layer{a}{\ell}=\layer{a}{-\ell}$.
\end{defn}

\begin{lem}
$a\mapsto\minus a$ is a negation map on any ELT ring $\R$.
\end{lem}
\begin{proof}
Straightforward from the definition.
\end{proof}

\begin{rem}
When dealing with ELT algebra, the relation $\symmdash$ is denoted $\vDash$.
\end{rem}

We point out some important elements in any ELT ring $\R$:
\begin{enumerate}
\item $\layer{0}{1}$, which is the multiplicative identity of $\R$.
\item $\layer{0}{0}$, which is idempotent to both operations of $\R$.
\item $\layer{0}{-1}$, which has the role of ``$-1$'' in our theory.
\end{enumerate}

Since ELT algebras lack an additive identity element, we adjoin one, denoted $0_R$ (denoted $-\infty$ in \cite{Blachar2016, Blachar2016b}), and we denote $\overline{\R}=\R\cup\left\{0_R\right\}$. This endows $\overline{\R}$ with an antiring structure.
\section{Modules over Semirings with a Negation Map}
Modules over rings have an important role in the classical theory. In the classical theory, a module~$M$ over a ring $R$ is an abelian group, which has a scalar multiplication operation by elements from~$R$. If one takes $R$ to be a semiring rather than a ring, then a semimodule $M$ is a commutative monoid, which has a scalar multiplication operation.\\

Since semirings with a negation map are not rings, modules over them lack some basic properties of modules over rings. The study of modules can be found in \cite[Chapters 14-18]{Golan1999}.\\

However, knowing that we are working over a semiring with a negation map instead of a general semiring, we note some unique properties of the module. For example, over a semiring with a negation map we have the notion of quasi-zeros, i.e.\ elements of the form $x+\minus x$, which play the role of the classical zero.\\

In this part, we look into modules over semirings with a negation map. The main issue which we will deal with is different definitions of base. Following \cite{Izhaki2010}, we will give four definitions:
\begin{enumerate}
\item \textbf{d-base} -- a maximal linearly independent set, \Dref{def:d-base}.
\item \textbf{s-base} -- a minimal spanning set, \Dref{def:s-base}.
\item \textbf{d,s-base} -- an s-base which is also a d-base, \Dref{def:d,s-base}.
\item \textbf{base} -- defined as the classical base, \Dref{def:base}.
\end{enumerate}

We will point out the properties satisfied by modules possessing these bases, and we will examine the differences between these definitions.
\subsection{The Surpassing Relation for Modules}

We will now give analogous definitions of some concepts defined previously for semirings with a negation map.

\begin{defn}
Let $R$ be a semiring, and let $M$ be a module with a negation map. An element of the form $x+\minus x$ for some $x\in M$ is called a \textbf{quasi-zero} (or a \textbf{balanced element}, as in \cite{Akian2008}). We denote $x^\circ=x+\minus x$. The \textbf{submodule of quasi-zeros} is
$$\zeroset{M}=\left\{x^\circ\mid x\in M\right\}$$
\end{defn}

If $R$ has a negation map, and the negation map of $M$ coincides with the induced negation map, then $\zeroset{M}=\zero M$.

\begin{defn}
Let $R$ be a semiring, and let $M$ be an $R$-module with a negation map. We define a relation $\symmdash$ on $M$ in the following way:
$$x\symmdash y\Longleftrightarrow\exists z\in\zeroset M:x=y+z$$
If $x\symmdash y$, we say that $x$ \textbf{surpasses} $y$.
\end{defn}

\begin{example}
We refer to some of the examples of modules given here, and demonstrate the meaning of the surpassing relation in these modules.
\begin{enumerate}
\item In $R$, the surpassing relation of the module coincides with the surpassing relation of the semiring with a negation map.
\item In $R^I$, the surpassing relation means surpassing componentwise.
\item As a special case of the free module, in $M_{m\times n}\left(R\right)$ the surpassing relation is equivalent to surpassing componentwise.
\item As another special case of the free module, in $R\left[\lambda\right]$, the surpassing relation means surpassing of each coefficient of the polynomials.
\end{enumerate}
\end{example}

\begin{lem}\label{lem:symmdash-is-ref-and-trans}
$\symmdash$ is a reflexive and transitive relation on any $R$-module with a negation map. However, it may not be antisymmetric, as demonstrated in \cite[Example 4.12]{Akian2008}.
\end{lem}
\begin{proof}
Let $M$ be an $R$-module with a negation map.
\begin{enumerate}
\item If $x\in M$, then $x=x+\zeroM$, and $\zeroM\in\zeroset M$; so $x\symmdash x$.
\item If $x,y,z\in M$ satisfy $x\symmdash y\symmdash z$, then $x=y+z_{1},y=z+z_{2}$, where $z_{1},z_{2}\in\zeroset M$. So
    $$x=y+z_{1}=z+z_{1}+z_{2}$$
    and $z_{1}+z_{2}\in\zeroset M$, implying $x\symmdash z$.
\end{enumerate}
\end{proof}

Although in general $\symmdash$ may not be a partial order relation, we give a necessary and sufficient condition for it to be one, and we point out a specific case in which it is a partial order relation.

\begin{lem}\label{lem:symmdash-is-partial-order}
The following are equivalent for an $R$-module $M$ with a negation map:
\begin{enumerate}
  \item $\forall x\in M\,\forall z_1,z_2\in\zeroset{M}:x=x+z_1+z_2\Rightarrow x=x+z_1\,\wedge\, x=x+z_2$.
  \item $\forall x\in M\,\forall z_1,z_2\in\zeroset{M}:x=x+z_1+z_2\Rightarrow x=x+z_1\,\vee\, x=x+z_2$.
  \item $\symmdash$ is a partial order relation on $M$.
\end{enumerate}
\end{lem}
\begin{proof}
$1\Rightarrow 2$ is trivial.\\

$2\Rightarrow 3$: Assume that condition $2$ holds. By \Lref{lem:symmdash-is-ref-and-trans}, it is enough to show that $\symmdash$ is antisymmetric. Let $x,y\in M$ such that $x\symmdash y$ and $y\symmdash x$. Then there are $z_1,z_2\in\zeroset{M}$ such that $x=y+z_1$ and $y=x+z_2$. Therefore, $x=x+z_1+z_2$. By condition 2, either $x=x+z_2=y$ (in which case we are done) or $x=x+z_1$. In the latter case,
$$x=x+z_1+z_2=x+z_2=y$$
which proves that $\symmdash$ is a partial order relation on $M$.\\

$3\Rightarrow 1$: Assume that $\symmdash$ is a partial order relation on $M$, and let $x\in M$ and $z_1,z_2\in\zeroset{M}$ such that $x=x+z_1+z_2$. Write $y=x+z_1$. By definition, $y\symmdash x$. On the other hand,
$$x=x+z_1+z_2=y+z_2\symmdash y$$
Since $\symmdash$ is antisymmetric, $x=y=x+z_1$, and thus also $x=x+z_1+z_2=x+z_2$.
\end{proof}

\begin{cor}\label{cor:R-is-idempotent}
Let $R$ be a semiring with a negation map, and let $M$ be an $R$-module (with the induced negation map). The following are equivalent:
\begin{enumerate}
  \item $\zero$ is idempotent.
  \item $\zeroset{R}$ is idempotent.
  \item $\zeroset{M}$ is idempotent for every $R$-module $M$.
\end{enumerate}
In this case, $\symmdash$ is a partial order relation on every $R$-module $M$.
\end{cor}
\begin{proof}
$1\Rightarrow 2$: Let $\alpha^\circ\in\zeroset{R}$. Then
$$\alpha^\circ+\alpha^\circ=\zero\cdot\alpha+\zero\cdot\alpha=\left(\zero+\zero\right)\alpha=\zero\alpha=\alpha^\circ$$

$2\Rightarrow 3$: Assume that $\zeroset{R}$ is idempotent, let $M$ be an $R$-module, and let $x^\circ\in\zeroset{M}$. Then
$$x^\circ+x^\circ=\zero x+\zero x=\left(\zero+\zero\right)x=\zero x=x^\circ$$

$3\Rightarrow 1$: Trivial.\\

It is left to prove that in this case, if $M$ is an $R$-module then $\symmdash$ is a partial order relation on~$M$. We prove condition $2$ in \Lref{lem:symmdash-is-partial-order}. Indeed, if $x\in M$ and $z_1,z_2\in\zeroset{M}$ satisfy $x=x+z_1+z_2$, then
$$x=x+z_1+z_2=x+z_1+z_2+z_2=x+z_2$$
\end{proof}

\begin{example}
If $R$ is a ring with $\minus a=-a$, a supertropical algebra with $\minus a=a$ or an ELT ring with $\minus a=\layer{0}{-1}a$, then $\symmdash$ is a partial order relation on every $R$-module.
\end{example}

\begin{example}
If $\symmdash$ is a partial order relation on $R$, the induced surpassing relation on an $R$-module $M$ might not be a partial order relation on $M$. For example, consider $R=\mathbb{N}_0$ with the usual addition and multiplication and the negation map $\minus a=a$. Then $\mathbb{N}_0^\circ=2\mathbb{N}_0$, and thus for $m,n\in\mathbb{N}_0$,
$$m\symmdash n\Leftrightarrow \exists k\in\mathbb{N}_0:m=n+2k$$
This is clearly a partial order relation.\\

However, consider the $R$-module $M=\mathbb{Z}$. The induced negation map is $\minus a=a$, and thus $\mathbb{Z}^\circ=2\mathbb{Z}$. Thus, for $m,n\in\mathbb{Z}$,
$$m\symmdash n\Leftrightarrow \exists k\in\mathbb{Z}:m=n+2k$$
which is not only not a partial order relation, but an equivalence relation.
\end{example}

We will later see that if $\symmdash$ is not a partial order, we can ``enforce'' it to be one, by taking the module modulo the congruence it generates (see \Dref{def:symmdash-is-partial-order}). \\

We note another property of $\zeroset{M}$:

\begin{lem}\label{lem:surpass-zero-is-zero}
If $x\in\zeroset{M}$ and $y\symmdash x$, then $y\in\zeroset{M}$.
\end{lem}
\begin{proof}
Since $y\symmdash x$, there exists $z\in\zeroset{M}$ such that $y=x+z$. The result follows, since $\zeroset{M}$ is a submodule of $M$.
\end{proof}

\subsection{Basic Definitions for Modules}
\subsubsection{Spanning Sets}

\begin{defn}
Let $R$ be a semiring with a negation map, let $M$ be an $R$-module, and let $S\subseteq M$ be a subset of $M$. The \textbf{$R$-module spanned by} \textbf{$S$}, denoted $\Span\left(S\right)$, is
$$\Span\left(S\right)=\left\{\sum_{i=1}^k\alpha_i x_i\in M\middle| k\in\mathbb{N},\alpha_{1},\dots,\alpha_{k}\in R,x_{1},\dots,x_{k}\in S\right\}$$
\end{defn}

It is easy to see that $S\subseteq\Span\left(S\right)$, and that $\textrm{Span}\left(S\right)\subseteq M$ is also an $R$-module with respect to the induced operations.

\begin{defn}
Let $R$ be a semiring with a negation map, and let $M$ be an $R$-module. We say that a subset $S\subseteq M$ is a \textbf{spanning set} of $M$, if $\Span\left(S\right)=M$.
\end{defn}

\begin{defn}
A module with a finite spanning set is called \textbf{finitely generated}.
\end{defn}

\subsubsection{Congruences and the First Isomorphism Theorem}

Since the usual construction of quotient modules using ideals fails over semirings, we use the language of congruences. In this subsubsection, we study congruences of modules over semirings with a negation map. Since this is a special case of congruences in the context of universal algebra, as written in \cite{Golan1999}, we mainly cite some known facts.

\begin{defn}\label{def:modules-hom}
Let $R$ be a semiring, and let $M_{1},M_{2}$ be two $R$-modules with a negation map. An \textbf{$R$-module homomorphism} is a function $\varphi:M_{1}\rightarrow M_{2}$, which satisfies:
\begin{enumerate}
\item $\forall x,y\in M_{1}:\varphi\left(x+y\right)=\varphi\left(x\right)+\varphi\left(y\right)$.
\item\label{itm:hom-commutes-with-scalars} $\forall\alpha\in R,\,\forall x\in M_{1}:\varphi\left(\alpha x\right)=\alpha\varphi\left(x\right)$.
\item\label{itm:hom-commutes-with-minus} $\forall x\in M_1:\varphi\left(\minus x\right)=\minus\varphi\left(x\right)$.
\end{enumerate}
\end{defn}

We note that if $R$ has a negation map, and the negation maps on $M_1$ and $M_2$ are the induced negation maps, then condition \ref{itm:hom-commutes-with-minus} follows from condition \ref{itm:hom-commutes-with-scalars}, since
$$\varphi\left(\minus x\right)=\varphi\left(\left(\minus\one\right)x\right)=\left(\minus\one\right)\varphi\left(x\right)=\minus\varphi\left(x\right).$$

\begin{rem}\label{rem:mod-zeros-sent-zeros}
$\varphi\left(\zeroset{M_1}\right)\subseteq\zeroset{M_2}$, since
$$\varphi\left(x^\circ\right)=\varphi\left(x+\minus x\right)=\varphi\left(x\right)+\minus\varphi\left(x\right)\in\zeroset{M_2}$$
\end{rem}

\begin{defn}
Let $R$ be a semiring, and let $M$ be an $R$-module with a negation map. An equivalence relation $\sim$ on $M$ is called a \textbf{module congruence}, if:
\begin{enumerate}
\item $\forall x_{1},x_{2},y_{1},y_{2}\in M:\, x_{1}\sim x_{2}\,\wedge\, y_{1}\sim y_{2}\Rightarrow x_{1}+y_{1}\sim x_{2}+y_{2}$.
\item\label{itm:cong-respects-scalar} $\forall\alpha\in R\,\forall x,y\in M:\, x\sim y\Rightarrow\alpha x\sim\alpha y$.
\item\label{itm:cong-respects-minus} $\forall x,y\in M:x\sim y\Rightarrow\minus x\sim\minus y$.
\end{enumerate}
The \textbf{kernel} of $\sim$ is
$$\kersim_{\sim}=\left\{ x\in M\mid\exists y\in\zeroset M:x\sim y\right\}$$
\end{defn}

Again, if the negation map on $M$ is induced from a negation map from $R$, condition \ref{itm:cong-respects-minus} follows from condition \ref{itm:cong-respects-scalar}. Therefore, the negation map $\minus$ of $M$ a negation map $\minus$ on $\quo{M}{\sim}$ by $\minus\left[x\right]=\left[\minus x\right]$. It is easy to see that this negation map and the induced negation map from $R$ coincide, since
$$\minus\left[x\right]=\left[\minus x\right]=\left[\left(\minus\one\right)x\right]=\left(\minus\one\right)\left[x\right]$$

\begin{lem}
$\quo M{\sim}$ is an $R$-module.
\end{lem}

\begin{lem}
$\kersim_{\sim}$ is a submodule of $M$.
\end{lem}
\begin{proof}
If $x_{1},x_{2}\in\kersim_{\sim}$, take $y_{1},y_{2}\in\zeroset M$
such that $x_{1}\sim y_{1}$ and $x_{2}\sim y_{2}$; then $x_{1}+x_{2}\sim y_{1}+y_{2}$,
and $y_{1}+y_{2}\in\zeroset M$. Thus, $x_{1}+x_{2}\in\kersim_{\sim}$.\\

Now, if $\alpha\in R,x\in\kersim_{\sim}$, take $y\in\zeroset M$ such
that $x\sim y$; then $\alpha x\sim\alpha y$, and $\alpha y\in\zeroset M$.
Thus, $\alpha x\in\kersim_{\sim}$.
\end{proof}

\begin{lem}
Let $\rho:M\to\quo{M}{\sim}$ be the canonical map. Then
$$\zeroset{\left(\quo{M}{\sim}\right)}=\rho\left(\kersim_{\sim}\right)$$
\end{lem}
\begin{proof}
$$\rho\left(x\right)\in\zeroset{\left(\quo{M}{\sim}\right)}\Longleftrightarrow \rho\left(x\right)=\zero\rho\left(x\right)=\rho\left(\zero x\right)\Longleftrightarrow x\sim\zero x\Longleftrightarrow x\in\kersim_{\sim}$$
\end{proof}

We recall the first isomorphism theorem:

\begin{lem}\label{lem:hom-induces-congr}
Let $\varphi:M_{1}\to M_{2}$ be an $R$-module homomorphism. Then
$$x\sim y\Leftrightarrow\varphi\left(x\right)=\varphi\left(y\right)$$
is a module congruence on $M_{1}$.
\end{lem}

\begin{thm}[The First Isomorphism Theorem]\label{thm:first-iso-thm}
Let $R$ be a semiring with a negation map, and let $M_{1},M_{2}$ be $R$-modules. If $\varphi:M_{1}\to M_{2}$ is an $R$-module homomorphism, then there exists a module congruence $\sim$ on $M_{1}$ such that
$$\quo{M_{1}}{\sim}\cong\Im\varphi$$
\end{thm}

We return to $\symmdash$ to demonstrate how we can ``enforce'' $\symmdash$ to be a partial order on $M$.

\begin{defn}\label{def:symmdash-is-partial-order}
Let $R$ be a semiring, and let $M$ be an $R$-module with a negation map. Define a relation~$\symmeq$ on $M$ as
$$a\symmeq b\Leftrightarrow a\symmdash b\,\wedge\, b\symmdash a$$
\end{defn}
\begin{rem}
$\symmeq$ is a congruence on $M$, and the $R$-module $\quo{M}{\symmeq}$ is partially ordered by the induced negation map $\minus\left[x\right]=\left[\minus x\right]$.
\end{rem}

\subsection{$\asymmdash$-morphisms}

As we will see more when dealing with Lie algebras with a negation map, we cannot always construct functions that will preserve every operation of the Lie algebra. We will now define the notion of $\asymmdash$-morphisms, as also defined in \cite[Section 8.2]{Rowen2016}.

\begin{defn}
Let $R$ be a semiring, and let $M_{1},M_{2}$ be two $R$-modules with a negation map. A \textbf{$\asymmdash$-morphism} is a function $\varphi:M_1\to M_2$, which satisfies:
\begin{enumerate}
\item $\forall x,y\in M_{1}:\varphi\left(x+y\right)\asymmdash\varphi\left(x\right)+\varphi\left(y\right)$.
\item $\forall\alpha\in R,\,\forall x\in M_{1}:\varphi\left(\alpha x\right)\asymmdash\alpha\varphi\left(x\right)$.
\item $\forall x\in M_1:\varphi\left(\minus x\right)=\minus\varphi\left(x\right)$.
\item $\varphi\left(\zeroset{M_1}\right)\subseteq\zeroset{M_2}$.
\end{enumerate}
\end{defn}

Our purpose now will be to formulate a version of the First Isomorphism Theorem for $\asymmdash$-morphisms. \\

Assume that $\varphi:M_1\to M_2$ is a \textit{surjective} $\asymmdash$-morphism. We define the equivalence relation $\sim$ (which is not necessarily a congruence) on $M_1$ as
$$x\sim y\Longleftrightarrow\varphi\left(x\right)=\varphi\left(y\right)$$
We now wish to define addition and scalar multiplication on $\quo{M_1}{\sim}$. The usual definition (adding or scalar multiplying the representatives) will no longer work, because $\sim$ may not be a congruence; so we define
\begin{eqnarray*}
  \left[x\right]+\left[y\right] &=& \left\{z\in M_1\middle|\varphi\left(z\right)=\varphi\left(x\right)+\varphi\left(y\right)\right\}=\varphi^{-1}\left(\varphi\left(x\right)+\varphi\left(y\right)\right) \\
  \alpha\cdot\left[x\right] &=& \left\{z\in M_1\middle|\varphi\left(z\right)=\alpha\cdot\varphi\left(x\right)\right\}=\varphi^{-1}\left(\alpha\cdot\varphi\left(x\right)\right)
\end{eqnarray*}
and the natural negation map $\minus\left[x\right]=\left[\minus x\right]$. \\

The usual verifications show that:

\begin{lem}
$\quo{M_1}{\sim}$ is an $R$-module with the above operations. The quasi-zeros are equivalence classes of the form $\left[x\right]$, where $x\in\ker\varphi$.
\end{lem}

We define the usual projection map $\rho:M_1\to\quo{M_1}{\sim}$ as $\rho\left(x\right)=\left[x\right]$.

\begin{lem}
$\rho$ is a $\asymmdash$-morphism.
\end{lem}
\begin{proof}~
\begin{enumerate}
  \item Let $x,y\in M_1$. Since $\varphi\left(x+y\right)\asymmdash\varphi\left(x\right)+\varphi\left(y\right)$, and since $\varphi$ is injective, there is some $z\in M_1$ such that
      $$\varphi\left(x+y\right)+\varphi\left(z\right)=\varphi\left(x\right)+\varphi\left(y\right)$$
      and $\varphi\left(z\right)\in\zeroset{M_2}$, that is, $z\in\ker\varphi$. Hence
      $$\rho\left(x+y\right)=\left[x+y\right]\asymmdash\left[x+y\right]+\left[z\right]=\left[x\right]+\left[y\right]=\rho\left(x\right)+\rho\left(y\right)$$
  \item Similarly as 1.
  \item Given $x\in M_1$, it is easily seen that
      $$\rho\left(\minus x\right)=\left[\minus x\right]=\minus\left[x\right]=\minus\rho\left(x\right)$$
  \item This is obvious.
\end{enumerate}
\end{proof}

We now get a version of the First Isomorphism Theorem:

\begin{thm}[The First Isomorphism Theorem for $\asymmdash$-morphisms]
Let $R$ be a semiring, let $M_{1},M_{2}$ be two $R$-modules with a negation map, and let $\varphi:M_1\to M_2$ be a $\asymmdash$-morphism. Let $\sim$ be the equivalence relation defined above, and $rho:M_1\to\quo{M_1}{\sim}$ the projection map. Then there exists a unique $R$-isomorphism $\hat{\varphi}:\quo{M_1}{\sim}\to M_2$ such that $\varphi=\hat{\varphi}\circ\rho$.
\end{thm}
\begin{proof}
We define $\hat{\varphi}:\quo{M_1}{\sim}$ as $\hat{\varphi}\left(\left[x\right]\right)=\varphi\left(x\right)$. The usual verifications as in the proof of the First Isomorphism Theorem prove that $\hat{\varphi}$ is an $R$-isomorphism, and that $\hat{\varphi}$ is uniquely defined.
\end{proof}

We define a similar concept for semirings with a negation map:
\begin{defn}
Let $R_1$ and $R_2$ be semirings with a negation map. A \textbf{$\asymmdash$-morphism} is a function $\varphi:R_1\to R_2$ such that the following properties hold:
\begin{enumerate}
  \item $\forall\alpha,\beta\in R_1:\varphi\left(\alpha+\beta\right)\asymmdash\varphi\left(\alpha\right)+\varphi\left(\beta\right)$.
  \item $\forall\alpha,\beta\in R_1:\varphi\left(\alpha\cdot\beta\right)\asymmdash\varphi\left(\alpha\right)\cdot\varphi\left(\beta\right)$.
  \item\label{itm:condition-of-symmdash-morphism} $\forall\alpha\in R_1:\minus\varphi\left(\alpha\right)=\varphi\left(\minus\alpha\right)$.
  \item $\varphi\left(0_{R_1}\right)=0_{R_2}$.
  \item $\varphi\left(1_{R_1}\right)=1_{R_2}$.
\end{enumerate}
\end{defn}

\subsection{Lifting a Module Over a Semiring with a Negation Map}\label{sec:Lifts-Semiring-Mod}

\subsubsection{Lifting a semiring with a negation map}

When dealing with tropical algebra, we had a powerful tool -- Puiseux series. This tool enables one to use known results in the classical theory to prove tropical results. In this section, we attempt to give a similar construction for an arbitrary semiring with a negation map.

\begin{defn}
Let $R$ be a semiring with a negation map. For any subset $A\subseteq R$, define
$$\overline{A}=\left\{\beta\in R\mid\exists\alpha\in A:\beta\symmdash\alpha\right\}$$
If $\overline{A}=R$, we say that $A$ is \textbf{$\asymmdash$-dense} in $R$.
\end{defn}

For example, $\overline{\nzeroset{R}}=R$. We note that this defines a topology on $R$; however, this topology is usually not even $T_1$, since, for example, $\overline{\left\{0_R\right\}}=\zeroset{R}$.

\begin{defn}
Let $R$ be a semiring with a negation map. A \textbf{lift} of $R$ is a ring $\widehat{R}$ and a map $\widehat{\varphi}:\widehat{R}\to R$, such that the following properties hold:
\begin{enumerate}
  \item $\widehat{\varphi}$ is a $\asymmdash$-morphism, where the negation map on $\widehat{R}$ is $\minus\widehat{\alpha}=-\widehat{\alpha}$.
  \item $\Im\widehat{\varphi}\subseteq R^{\vee}$ is $\asymmdash$-dense in $R$.
  \item $\widehat{\varphi}\left(\widehat{\alpha}\right)=0_R\Longleftrightarrow\widehat{\alpha}=0_{\widehat{R}}$
\end{enumerate}
\end{defn}

\begin{example}
We give several examples of lifts.
\begin{enumerate}
\item If $R$ is a ring with the negation map $\minus a=-a$, then the identity map $R\to R$ is a lift of~$R$.
\item Given an ELT algebra $\R=\ELT{\L}{\F}$, $\widehat{\R}=\puis{\L}$ with the EL tropicalization map is a lift of $\R$, as defined and proved in \cite[Lemma 0.4]{Blachar2016}.
\item Although the intuition is that lifts should be ``very big'', this example proves that this is not always the case. Consider the semiring $\left(\N_0,+,\cdot\right)$ with the negation map $\minus a=a$. Then~$\quo{\Z}{2\Z}$ is a lift of $\N_0$, with the map defined by $\widehat{\varphi}\left(\bar{0}\right)=0$ and $\widehat{\varphi}\left(\bar{1}\right)=1$.
\end{enumerate}
\end{example}

We now prove the following theorem:

\begin{thm}
Let $R$ be an antiring with a negation map. Then $R$ possesses a lift.
\end{thm}
\begin{proof}
We shall first construct the lift. We denote by $A$ the following set:
$$A=\left\{a_\alpha\mid\alpha\in R^\vee\right\}$$
(We use the notation $a_\alpha$ rather than $\alpha$ to distinguish the elements of $A$ and those of $R^\vee$).\\

We define a multiplication on the elements of $A$ as follows:
$$a_\alpha\cdot a_\beta=\left\{\begin{matrix}a_{\alpha\beta},&\alpha\beta\in R^\vee\\a_{0_R},&\mathrm{Otherwise}\end{matrix}\right.$$
This multiplication endows $A$ with a monoid structure.\\

We also denote $\widehat{R}=\mathbb{Z}\left[A\right]$, meaning
$$\widehat{R}=\left\{\sum_{i=1}^k n_i a_{\alpha_i}\middle|n_i\in\mathbb{Z}, \alpha_i\in R^\vee\right\}$$
Since $A$ is a monoid, $\widehat{R}$ is a ring.\\

We are left with defining the lifting map $\widehat{\varphi}:\widehat{R}\to R$. We define an action of $\left\{-1,0,1\right\}$ on $R$ by
$$\begin{matrix}1\cdot\alpha=\alpha,&0\cdot\alpha=0_R,&\left(-1\right)\cdot\alpha=\minus\alpha\end{matrix}$$
The map $\widehat{\varphi}:\widehat{R}\to R$ is defined by $\widehat{\varphi}\left(0_{\widehat{R}}\right)=0_R$ and
$$\widehat{\varphi}\left(\sum_{i=1}^{k}n_i a_{\alpha_i}\right)=\sum_{i=1}^{k}\left|n_i\right|\left(\sign\left(n_i\right)\cdot\alpha_i\right)$$
where $\displaystyle{\sum_{i=1}^{k}n_i a_{\alpha_i}}$ is in its reduced representation, i.e.\ $\alpha_i\neq\alpha_j$ for $i\neq j$, and $\sign\left(n_i\right)\cdot\alpha_i$ is calculated by the above action.\\

We shall now prove that $\left(\widehat{R},\widehat{\varphi}\right)$ is a lift of $R$. We first prove that $\widehat{\varphi}$ is a $\asymmdash$-morphism.

\begin{enumerate}
  \item It is easy to see that if $\alpha_i\neq\alpha_j$ whenever $i\neq j$,
  $$\widehat{\varphi}\left(\sum_{i=1}^{k}n_ia_{\alpha_i}\right)=\sum_{i=1}^{k}\widehat{\varphi}\left(n_i a_{\alpha_i}\right)$$
  Therefore we only have to prove that $\widehat{\varphi}\left(m a_{\alpha}\right)+\widehat{\varphi}\left(n a_{\alpha}\right)\symmdash\widehat{\varphi}\left(ma_{\alpha}+na_{\alpha}\right)$ where $m,n\neq 0$. If $\sign\left(m\right)=\sign\left(n\right)$, then
  \begin{eqnarray*}
    \widehat{\varphi}\left(ma_{\alpha}\right)+\widehat{\varphi}\left(na_{\alpha}\right) &=& \left|m\right|\left(\sign\left(m\right)\cdot\alpha\right)+ \left|n\right|\left(\sign\left(n\right)\cdot\alpha\right)= \left|m+n\right|\left(\sign\left(m+n\right)\cdot\alpha\right)= \\
     &=& \widehat{\varphi}\left(\left(m+n\right)a_{\alpha}\right)= \widehat{\varphi}\left(ma_{\alpha}+na_{\alpha}\right)
  \end{eqnarray*}
  Now, suppose $\sign\left(m\right)\neq\sign\left(n\right)$. Without loss of generality, we assume $n<0<m$ and $-n<m$ (the other cases are proved similarly). Thus,
  \begin{eqnarray*}
    \widehat{\varphi}\left(ma_{\alpha}\right)+\widehat{\varphi}\left(na_{\alpha}\right) &=& m\alpha+ \left(-n\right)\left(\minus\alpha\right)= \left(m+n\right)\alpha+\left(-n\right)\alpha+\left(-n\right)\left(\minus\alpha\right)= \\
     &=& \left(m+n\right)\alpha+\left(-n\right)\alpha^\circ\symmdash\left(m+n\right)\alpha= \widehat{\varphi}\left(\left(m+n\right)a_{\alpha}\right)= \\ &=&\widehat{\varphi}\left(ma_{\alpha}+na_{\alpha}\right)
  \end{eqnarray*}
  as we needed to prove.
  \item Let us begin by observing that
      $$\widehat{\varphi}\left(a_\alpha\right)\widehat{\varphi}\left(a_\beta\right)\symmdash\widehat{\varphi}\left(a_{\alpha\beta}\right)$$
      (Since if $\alpha\beta\in R^{\vee}$, there is equality; otherwise, the LHS is $\alpha\beta\in\zeroset{R}$, whereas the RHS is $0_R$).
      Now,
      \begin{eqnarray*}
        \widehat{\varphi}\left(\sum_{i=1}^{k}m_i a_{\alpha_i}\right)\cdot\widehat{\varphi}\left(\sum_{j=1}^{\ell}n_j a_{\beta_j}\right) &=& \left(\sum_{i=1}^{k}m_i\alpha_i\right)\cdot\left(\sum_{j=1}^{\ell}n_j\beta_j\right)=\\
         &=& \sum_{\alpha_i\beta_j\in R^\vee}\left(m_i n_j\right)\alpha_i\beta_j+\sum_{\alpha_i\beta_j\notin R^\vee}\left(m_i n_j\right)\alpha_i\beta_j\symmdash\\
         &\symmdash&\sum_{\alpha_i\beta_j\in R^\vee}\left(m_i n_j\right)\alpha_i\beta_j= \widehat{\varphi}\left(\sum_{\alpha_i\beta_j\in R^\vee}\left(m_i n_j\right)a_{\alpha_i \beta_j}\right)\\
         &=&\widehat{\varphi}\left(\left(\sum_{i=1}^{k}m_i a_{\alpha_i}\right)\cdot\left(\sum_{j=1}^{\ell}n_j a_{\beta_j}\right)\right)
      \end{eqnarray*}
  \item $$\widehat{\varphi}\left(-\sum_{i=1}^{k}n_ia_{\alpha_i}\right)= \sum_{i=1}^{k}\left|-n_i\right|\left(\sign\left(-n_i\right)\cdot\alpha_i\right)= \minus\sum_{i=1}^{k}\left|n_i\right|\left(\sign\left(n_i\right)\cdot\alpha_i\right)= \minus\widehat{\varphi}\left(\sum_{i=1}^{k}n_ia_{\alpha_i}\right)$$
  \item $\widehat{\varphi}\left(0_{\widehat{R}}\right)=0_R$ by definition.
  \item Noting that $1_{\widehat{R}}=1\cdot a_{1_R}$, $\widehat{\varphi}\left(1_{\widehat{R}}\right)=1_R$.
\end{enumerate}

This proves that $\widehat{\varphi}$ is a $\asymmdash$-morphism. Since $\Im\widehat{\varphi}=R^\vee$, we are left to prove that $\widehat{\varphi}\left(\widehat{\alpha}\right)=0_R$ if and only if $\widehat{\alpha}=0_{\widehat{R}}$. But this follows from the fact that $R$ is an antiring.\\

Hence $\left(\widehat{R},\widehat{\varphi}\right)$ is a lift of $R$.
\end{proof}

\subsubsection{Lifting a module}
We move towards lifting a module. We use again the concept of $\asymmdash$-density:
\begin{defn}
Let $R$ be a semiring, let $M$ be an $R$-module with a negation map. For any subset $S\subseteq M$, define
$$\overline{S}=\left\{y\in M\mid\exists x\in S:y\symmdash x\right\}$$
If $\overline{S}=M$, we say that $S$ is \textbf{$\asymmdash$-dense} in $M$.
\end{defn}

\begin{defn}
Let $R$ be a semiring with a negation map with a lift $\left(\widehat{R},\widehat{\varphi}\right)$, and let $M$ be an $R$-module. A \textbf{lift} of $M$ is an $\widehat{R}$-module $\widehat{M}$ and a map $\widehat{\psi}:\widehat{M}\to M$ such that the following properties hold:
\begin{enumerate}
  \item $\forall x_1,x_2\in \widehat{M}:\widehat{\psi}\left(x_1\right)+\widehat{\psi}\left(x_2\right)\symmdash\widehat{\psi}\left(x_1+ x_2\right)$.
  \item $\forall\widehat{\alpha}\in\widehat{R}\,\forall x\in\widehat{M}:\widehat{\varphi}\left(\widehat{\alpha}\right)\widehat{\psi}\left(x\right)\symmdash\widehat{\psi}\left(\widehat{\alpha}\widehat{x}\right)$.
  \item $\widehat{\psi}\left(0_{\widehat{M}}\right)=0_M$.
  \item $\Im\widehat{\psi}$ is $\asymmdash$-dense in $M$.
\end{enumerate}
\end{defn}

\begin{thm}
If $R$ has a lift, then every $R$-module has a (free) lift. Furthermore, if the module is generated by $\mu$ elements (for some cardinal number $\mu$), it possesses a lift which is a free module generated by $\mu$ elements.
\end{thm}
\begin{proof}
Let $\left(\widehat{R},\widehat{\varphi}\right)$ be a lift of $R$. Let $S\subseteq M$ be a generating set. Define $\widehat{M}=\widehat{R}^S$, and define a function $\widehat{\psi}:\widehat{M}\to M$ by
$$\widehat{\psi}\left(\left(\widehat{r}_s\right)_{s\in S}\right)=\sum_{s\in S}\widehat{\varphi}\left(\widehat{r}_s\right)s$$

We now prove that $\left(\widehat{M},\widehat{\psi}\right)$ is a lift of $M$.
\begin{enumerate}
  \item Let $\left(\widehat{x}_s\right)_{s\in S},\left(\widehat{y}_s\right)_{s\in S}\in\widehat{M}$. Then
  \begin{eqnarray*}
    \widehat{\psi}\left(\left(\widehat{x}_s\right)_{s\in S}\right)+\widehat{\psi}\left(\left(\widehat{y}_s\right)_{s\in S}\right)&=&\sum_{s\in S}\widehat{\varphi}\left(\widehat{x}_s\right)s+\sum_{s\in S}\widehat{\varphi}\left(\widehat{y}_s\right)s=\sum_{s\in S}\left(\widehat{\varphi}\left(\widehat{x}_s\right)+\widehat{\varphi}\left(\widehat{y}_s\right)\right)s\symmdash\\
     &\symmdash& \sum_{s\in S}\widehat{\varphi}\left(\widehat{x}_s+\widehat{y}_s\right)s=\widehat{\psi}\left(\left(\widehat{x}_s\right)_{s\in S}+\left(\widehat{y}_s\right)_{s\in S}\right)
  \end{eqnarray*}
  \item Let $\widehat{\alpha}\in\widehat{R}$ and $\left(\widehat{x}_s\right)_{s\in S}\in\widehat{M}$. Then
  $$\widehat{\varphi}\left(\widehat{\alpha}\right)\widehat{\psi}\left(\left(\widehat{x}_s\right)_{s\in S}\right)= \widehat{\varphi}\left(\widehat{\alpha}\right)\sum_{s\in S}\widehat{\varphi}\left(\widehat{x}_s\right)s=\sum_{s\in S}\left(\widehat{\varphi}\left(\widehat{\alpha}\right)\widehat{\varphi}\left(\widehat{x}_s\right)\right)s\symmdash\sum_{s\in S}\widehat{\varphi}\left(\widehat{\alpha}\widehat{x}_s\right)s=\widehat{\psi}\left(\widehat{\alpha}\left(\widehat{x}_s\right)_{s\in S}\right)$$
  \item Follows immediately.
  \item Let $x\in M$. Write $x=\sum_{s\in S}\alpha_s s$. Since $\Im\widehat{\varphi}$ is $\asymmdash$-dense in $R$, for each $\alpha_s$ there is $\widehat{\alpha}_s\in\widehat{R}$ such that $\alpha_s\symmdash\widehat{\varphi}\left(\widehat{\alpha}_s\right)$. If $\alpha_s=0_R$, we may choose $\widehat{\alpha}_s=0_{\widehat{R}}$. Therefore,
      $$x=\sum_{s\in S}\alpha_s s\symmdash\sum_{s\in S}\widehat{\varphi}\left(\widehat{\alpha}_s\right)s=\widehat{\psi}\left(\left(\widehat{\alpha}_s\right)_{s\in S}\right)$$
\end{enumerate}
\end{proof}

\begin{rem}\label{rem:Lift-of-R^n}
If $\left(\widehat{R},\widehat{\varphi}\right)$ is a lift of $R$, then $\left(\left(\widehat{R}\right)^n,\widehat{\psi}\right)$ is a lift of $R^n$, where
$$\left(\widehat{\psi}\left(\widehat{x}\right)\right)_i=\widehat{\varphi}\left(\left(\widehat{x}\right)_i\right)$$
I.e., $\widehat{\psi}$ applies $\widehat{\varphi}$ on each entry of the given vector. In this case, $\Im\widehat{\psi}=\left(R^{\vee}\right)^n$.
\end{rem}

We will later see theorems which hold over modules which have a lift, such as \Cref{cor:n+1-vec-in-Rn-are-dep}.

\begin{lem}\label{lem:lift-submodule}
Let $\widehat{N}\subseteq\widehat{M}$ be a submodule. Then $\overline{\widehat{\psi}\left(\widehat{N}\right)}$ is a submodule of $M$.
\end{lem}
\begin{proof}
Let $x,y\in\overline{\widehat{\psi}\left(\widehat{N}\right)}$ and let $\alpha,\beta\in R$. Take $\widehat{x},\widehat{y}\in\widehat{M}$ such that $x\symmdash\widehat{\psi}\left(\widehat{x}\right)$ and $y\symmdash\widehat{\psi}\left(\widehat{y}\right)$, and take $\widehat{\alpha},\widehat{\beta}\in\widehat{R}$ such that $\alpha\symmdash\widehat{\varphi}\left(\widehat{\alpha}\right)$ and $\beta\symmdash\widehat{\varphi}\left(\widehat{\beta}\right)$. Then
$$\alpha x+\beta y\symmdash\widehat{\varphi}\left(\widehat{\alpha}\right)\widehat{\psi}\left(\widehat{x}\right)+ \widehat{\varphi}\left(\widehat{\beta}\right)\widehat{\psi}\left(\widehat{y}\right)\symmdash \widehat{\psi}\left(\widehat{\alpha}\widehat{x}\right)+\widehat{\psi}\left(\widehat{\beta}\widehat{y}\right)\symmdash \widehat{\psi}\left(\widehat{\alpha}\widehat{x}+\widehat{\beta}\widehat{y}\right)$$
Since $\widehat{\alpha}\widehat{x}+\widehat{\beta}\widehat{y}\in N$, we are finished.
\end{proof}

\subsection{Linearly Independent Sets}

Up until now, most of our results were formulated and proved for a module with a negation map, where we didn't assume that the underlying semiring has a negation map. However, now that we are going to deal with linear dependency, we need the notion of quasi-zero scalars; hence, we assume for the rest of this section that our semiring has a negation map.

We first specialize our underlying semiring, to avoid the problem of ``quasi-zero divisors'':
\begin{defn}
A semiring with a negation map $R$ is called \textbf{entire}, if $R$ is commutative in both of the operations and if
$$\forall\alpha,\beta\in R:\alpha\beta\in\zeroset{R}\Rightarrow \alpha\in\zeroset{R}\vee\beta\in\zeroset{R}$$
\end{defn}

\begin{defn}
Let $R$ be an entire semiring with a negation map, and let $M$ be an $R$-module. A set $S\subseteq M$ is called \textbf{$\circ$-linearly dependent}, if
$$\exists k\in\N,\;\exists x_{1},\dots,x_{k}\in S,\;\exists\alpha_{1},\dots,\alpha_{k}\in \nzeroset{R}\setminus\left\{\minf\right\}:\alpha_{1}x_{1}+\cdots+\alpha_{k}v_{k}\in\zeroset M$$
$S$ is called \textbf{$\circ$-linearly independent}, if it is not linearly dependent.
\end{defn}

Since we are dealing with negation maps, we will omit the $\circ$ in $\circ$-linearly dependent and $\circ$-linearly independent.

\begin{lem}
Let $R$ be an entire semiring with a negation map, let $M$ be an $R$-module, and let $S\subseteq M$ be a linearly independent set. Then
$$S\cap\zeroset{M}=\emptyset$$
\end{lem}
\begin{proof}
Assume there exists some $x\in S$ such that $x\in\zeroset{M}$. Then the linear combination $\one x$ is quasi-zero, contradicting the fact that $S$ is linearly independent.
\end{proof}

\begin{lem}\label{lem:surpass-dep-is-dep}
Let $R$ be an entire semiring with a negation map, let $M$ be an $R$-module, and let $S=\left\{x_1,\dots,x_m\right\}$ be a linearly dependent subset of $M$. Assume that $S'=\left\{y_1,\dots,y_m\right\}\subseteq M$ satisfies $\forall i:y_i\symmdash x_i$. Then $S'$ is also linearly dependent.
\end{lem}
\begin{proof}
Assume that $\displaystyle{\sum_i\alpha_i x_i\in\zeroset{M}}$, where $\forall i:\alpha_i\in R^{\vee}$; then
$$\sum_i\alpha_i y_i\symmdash\sum_i\alpha_i x_i\in\zeroset{M}$$
implying $\displaystyle{\sum_i\alpha_i y_i\in\zeroset{M}}$, by \Lref{lem:surpass-zero-is-zero}.
\end{proof}

\begin{lem}\label{lem:surpass-implies-diff-zero}
Let $M$ be an $R$-module. If $x\symmdash y$, then $x+\minus y\in\zeroset{M}$.
\end{lem}
\begin{proof}
By definition, there is some $z\in\zeroset{M}$ such that $x=y+z$. Therefore,
$$x+\minus y=y+z+\minus y=\zero y+z\in\zeroset{M}$$
\end{proof}

\begin{lem}\label{lem:lin-comb-dep}
Let $R$ be an entire semiring with a negation map, and let $M$ be an $R$-module. Assume that ${\displaystyle y\symmdash\sum_{i=1}^{k}\alpha_{i}x_{i}}$. Then the set $\left\{ x_{1},\dots,x_{k},y\right\} $ is linearly dependent.
\end{lem}
\begin{proof}
Write
$$I=\left\{i=1,\dots,n\middle|\alpha_i\notin\zeroset{R}\right\}$$
We note that if $i\notin I$, then $\alpha_i\in\zeroset{R}$, and thus $\alpha_i x_i\in\zeroset{M}$. If $I=\varnothing$, then $y\in\zeroset{M}$, and thus $\left\{ x_{1},\dots,x_{k},y\right\} $ is linearly dependent. Therefore, we may assume that $I\neq\varnothing$, and thus
$$y\symmdash\sum_{i=1}^k\alpha_i x_i\symmdash\sum_{i\in I}\alpha_i x_i$$

By \Lref{lem:surpass-implies-diff-zero},
$$y+\sum_{i\in I}\left(\minus\alpha_i\right) x_i=y+\minus\sum_{i\in I}\alpha_i x_i\in\zeroset{M}.$$
We found a linear combination of some of the vectors in the set $\left\{x_{1},\dots,x_{k},x\right\}$, which is quasi-zero, and the coefficients are not quasi-zero, implying $\left\{x_{1},\dots,x_{k},x\right\}$ is linearly dependent.
\end{proof}

\begin{lem}\label{lem:sub-lin-ind-not-span}
Let $S\subseteq M$ be a linearly independent set. Then for all $x\in S$, $S\backslash\left\{x\right\}$ is not a spanning set of~$M$.
\end{lem}
\begin{proof}
Assume that $S\backslash\left\{ x\right\} $ is a spanning set of $M$. Then
$$\exists k\in\mathbb{N},\,\exists\alpha_{1},\dots,\alpha_{k}\in R,\,\exists x_{1},\dots,x_{k}\in S:\alpha_{1}x_{1}+\cdots+\alpha_{k}x_{k}=x=\one x$$
By \Lref{lem:lin-comb-dep}, $\left(S\backslash\left\{x\right\}\right)\cup\left\{x\right\}=S$ is linearly dependent, a contradiction. Thus, $S\backslash\left\{x\right\}$ is not a spanning set of $M$.
\end{proof}

\begin{cor}\label{cor:span-not-sub-ind}
Suppose $A\subseteq M$ is a linearly independent set, and that $B\subseteq M$ is a spanning set of $M$. If $B\subseteq A$, then $A=B$.
\end{cor}
\begin{proof}
If $A\neq B$, there exists $x\in A\setminus B$. Therefore, there exists a linear combination
$$x=\sum_{b\in B}\alpha_{b}b$$
By \Lref{lem:lin-comb-dep}, $B\cup\left\{x\right\}$ is linearly dependent. But $B\cup\left\{x\right\}\subseteq A$, contradicting the assumption that $A$ is linearly independent. Thus $A=B$.
\end{proof}

\subsection{d-bases and s-bases}
\subsubsection{d-bases}
\begin{defn}\label{def:d-base}
Let $R$ be an entire semiring with a negation map. A \textbf{d-base} (for dependence base) of an $R$-module $M$ is a maximal linearly independent subset of $M$.
\end{defn}

\begin{defn}
Let $R$ be an entire semiring with a negation map, and let $M$ be an $R$-module. The \textbf{rank} of $M$ is
$$\rk\left(M\right)=\max\left\{ \left|B\right|\middle| B\textrm{ is a d-base of }M\right\}$$
\end{defn}

\begin{lem}\label{lem:lin-ind-in-maximal}
Every linearly independent set is contained in some d-base.
\end{lem}
\begin{proof}
Let $M$ be an $R$-module, and let $S\subseteq M$ be a linearly independent set. Consider
$$\left\{S'\subseteq M\middle| S\subseteq S'\textrm{ is linearly independent}\right\}$$
This set satisfies Zorn's condition, and thus, has a maximal element $S'',$ which is a d-base of $M$.
\end{proof}

\subsubsection{d-bases of modules which possess a lift}

For this part, we fix an entire semiring with a negation map $R$ with a lift $\left(\widehat{R},\widehat{\varphi}\right)$.

\begin{lem}\label{lem:ELTrop-of-dep-is-dep}
Let $M$ be an $R$-module with a lift $\left(\widehat{M},\widehat{\psi}\right)$, and let $\widehat{x}_1,\dots,\widehat{x}_m\in\widehat{M}$ be vectors. If $\widehat{x}_1,\dots,\widehat{x}_m$ are linearly dependent, then $\widehat{\psi}\left(\widehat{x}_1\right),\dots,\widehat{\psi}\left(\widehat{x}_m\right)$ are also linearly dependent.
\end{lem}
\begin{proof}
$\widehat{x}_1,\dots,\widehat{x}_m$ are linearly dependent, so there are $\widehat{\alpha}_1,\dots,\widehat{\alpha}_m\in\widehat{R}$, not all are $0_{\widehat{R}}$, such that
$$\sum_{i=1}^m\widehat{\alpha}_i \widehat{x}_i=0_{\widehat{M}}$$
Applying $\widehat{\psi}$ on the equation, we get
$$\sum_{i=1}^{m}\widehat{\varphi}\left(\widehat{\alpha}_i\right)\widehat{\psi}\left(\widehat{x}_i\right)\symmdash\widehat{\psi}\left(\sum_{i=1}^{m}\widehat{\alpha}_i \widehat{x}_i\right)=\widehat{\psi}\left(0_{\widehat{M}}\right)=0_M$$
Using \Lref{lem:surpass-zero-is-zero}, we get
$$\sum_{i=1}^{m}\widehat{\varphi}\left(\widehat{\alpha}_i\right)\widehat{\psi}\left(\widehat{x}_i\right)\in\zeroset{M}$$
Since $\Im\widehat{\varphi}\subseteq\nzeroset{R}$, we get that for each $i$, $\widehat{\varphi}\left(\widehat{\alpha}_i\right)\in\nzeroset{R}$.
But we know that not all of the $\widehat{\alpha}_i$'s are $0_{\widehat{R}}$, and thus not all of the $\widehat{\varphi}\left(\widehat{\alpha}_i\right)$ are $0_R$, as required.
\end{proof}

\begin{cor}\label{cor:n+1-vec-in-Rn-are-dep}
If $R$ has a lift, than any $n+1$ vectors in $R^n$ are linearly dependent.
\end{cor}
\begin{proof}
Let $x_1,\dots,x_{n+1}\in R^n$ be any vectors. By \Rref{rem:Lift-of-R^n}), $R^n$ has a lift $\left(\left(\widehat{R}\right)^n,\widehat{\psi}\right)$. Let $x'_1,\dots,x'_{n+1}\in\Im\widehat{\psi}$ such that $x_i\symmdash x'_i$ for each $i$. Now, let $y_1,\dots,y_{n+1}\in\left(\widehat{R}\right)^n$ be some vectors such that for each $i$, $\widehat{\psi}\left(y_i\right)=x'_i$.\\

$\widehat{R}$ is a ring, hence any $n+1$ vectors in $\left(\widehat{R}\right)^n$ are linearly dependent; in particular, $y_1,\dots,y_{n+1}$ are linearly dependent. By \Lref{lem:ELTrop-of-dep-is-dep}, $x'_1,\dots,x'_{n+1}$ are linearly dependent. By \Lref{lem:surpass-dep-is-dep}, we are finished.
\end{proof}

\begin{thm}\label{thm:max-num-of-lin-ind-in-R^n-is-n}
If $R$ has a lift, then $\rk\left(R^{n}\right)=n$.
\end{thm}
\begin{proof}
By \Cref{cor:n+1-vec-in-Rn-are-dep}, any $n+1$ vectors in $R^{n}$ are linearly dependent; so $\rk\left(R^{n}\right)\leq n$. However, since $\left\{e_1,\dots,e_n\right\}$ are linearly dependent, $\rk\left(R^n\right)\ge n$, which together yield the desired equality.
\end{proof}

Recall the relation $\symmeq$ from \Dref{def:symmdash-is-partial-order}, defined as
$$a\symmeq b\Leftrightarrow a\symmdash b\,\wedge\, b\symmdash a$$

\begin{thm}\label{thm:n+1-vectors-when-R-mod-symmeq-has-lift}
Assume that $\quo{R}{\symmeq}$ has a lift (rather than $R$). Then any $n+1$ vectors in $R^n$ are linearly dependent.
\end{thm}
\begin{proof}
Let $x_1,\dots,x_{n+1}\in R^n$, and let $\left[x_1\right],\dots,\left[x_{n+1}\right]\in\quo{R^n}{\symmeq}$ be their equivalence classes under~ $\symmeq$. Noting that
$$\quo{R^n}{\symmeq}\cong\left(\quo{R}{\symmeq}\right)^n$$
by \Cref{cor:n+1-vec-in-Rn-are-dep}, there are $\alpha_1,\dots,\alpha_{n+1}\in R^{\vee}$, not all are $0_R$, such that
$$\left[\alpha_1 x_1+\cdots+\alpha_{n+1}x_{n+1}\right]=\alpha_1\left[x_1\right]+\cdots+\alpha_{n+1}\left[x_{n+1}\right]\in\zeroset{\left(\quo{R^n}{\symmeq}\right)}$$
Hence,
$$\left[\alpha_1 x_1+\cdots+\alpha_{n+1}x_{n+1}\right]=\zero\left[\alpha_1 x_1+\cdots+\alpha_{n+1}x_{n+1}\right]=\left[\zero\left(\alpha_1 x_1+\cdots+\alpha_{n+1}x_{n+1}\right)\right]$$
By the definition of $\symmeq$,
$$\alpha_1 x_1+\cdots+\alpha_{n+1}x_{n+1}\symmdash\zero\left(\alpha_1 x_1+\cdots+\alpha_{n+1}x_{n+1}\right)\in\zeroset{\left(R^n\right)}$$
By \Lref{lem:surpass-zero-is-zero},
$$\alpha_1 x_1+\cdots+\alpha_{n+1}x_{n+1}\in\zeroset{\left(R^n\right)}$$
as required.
\end{proof}

\begin{cor}
Let $R$ be a ring with some negation map. Then any $n+1$ vectors in $R^n$ are linearly dependent.
\end{cor}
\begin{proof}
In this case, $\quo{R}{\symmeq}$ is a ring, and the induced negation map is $\minus\left[a\right]=-\left[a\right]$; therefore, $\quo{R}{\symmeq}$ has a lift, and \Tref{thm:n+1-vectors-when-R-mod-symmeq-has-lift} can be applied.
\end{proof}

In the above cases, we have another corollary:

\begin{cor}
If $M\subseteq R^{n}$ is a submodule, then $\rk\left(M\right)\leq n$.
\end{cor}
\begin{proof}
Any d-base of $M$ is contained in a d-base of $R^{n}$, due to \Lref{lem:lin-ind-in-maximal}, whose order is at most~$n$.
\end{proof}

\subsubsection{s-bases}
\begin{defn}\label{def:s-base}
Let $R$ be an entire semiring with a negation map. An \textbf{s-base} (for spanning base) of an $R$-module $M$ is a minimal spanning subset of $M$.
\end{defn}

An easy corollary from the First Isomorphism Theorem is the following:

\begin{cor}\label{cor:fin-gen-module-cong-of-Rn}
Let $M$ be an $R$-module with a finite spanning set $S$, $\left|S\right|=n$. Then there exists a congruence $\sim$ on $R^{n}$, such that
$$M\cong\quo{R^{n}}{\sim}$$
\end{cor}

\begin{cor}\label{cor:span-set-at-least-rk-ele}
Any spanning set of an $R$-module $M$ must have at least $\rk\left(M\right)$ elements.
\end{cor}
\begin{proof}
By \Cref{cor:fin-gen-module-cong-of-Rn}, if $M=\Span\left\{ x_{1},\dots,x_{n}\right\} $, then $M\cong\quo{R^{n}}{\sim}$ for some congruence $\sim$.\\

Thus, any d-base $\left\{ y_{1},\dots,y_{m}\right\} $ of $M$ maps to a linearly independent set $\left\{\left[y'_{1}\right],\dots,\left[y'_{m}\right]\right\}$ of $\quo{R^{n}}{\sim}$. Then also $y'_{1},\dots,y'_{m}$ are linearly independent in $R^{n}$, and thus $m\leq n$.
\end{proof}

\subsubsection{Critical elements versus s-bases}
We define a similar concept to the concept of critical elements presented in \cite{Izhaki2010}:
\begin{defn}
Let $R$ be an entire semiring with a negation map, and let $M$ be an $R$-module.
\begin{enumerate}
  \item We define an equivalence relation $\sim$ on $M$, called \textbf{projective equivalence}, as the transitive closure of the following relation: we say that $x\sim y$ if there is an invertible $\alpha\in R^{\vee}$ such that $x=\alpha y$.
  \item We define the \textbf{equivalence class} of $x$ to be
      $$\left[x\right]_{\sim}=\left\{ y\in M\middle| y\sim x\right\}$$
  \item We say that $x\in M\backslash\zeroset M$ is \textbf{critical}, if there is no linear combination $x=x_{1}+x_{2}$ where $x_{1},x_{2}\in M\backslash\left[x\right]_{\sim}$.
  \item A \textbf{critical set} is a set of representatives of all the equivalence classes of $\sim$.
\end{enumerate}
\end{defn}

\begin{rem}
Any critical set is projectively unique.
\end{rem}

\begin{lem}\label{lem:critical-not-span-by-others}
Suppose $x\in M$ is critical. Then it is not spanned by $M\backslash\left[x\right]_{\sim}$.
\end{lem}
\begin{proof}
Assume that
$$x=\sum_{i=1}^{n}\alpha_{i}y_{i}$$
for $y_{i}\in M\backslash\left[x\right]_{\sim}$. $x$ is critical, thus $n>1$. Write $x_{1}=\alpha_{1}y_{1}$ and $x_{2}=\sum_{i=2}^{n}\alpha_{i}y_{i}$. By the definition of criticality, $x_{2}\in\left[x\right]_{\sim}$. We get a contradiction by induction on $n$.
\end{proof}

\begin{lem}\label{lem:s-base-contains-critical}
Suppose $S$ is an s-base of $M$. Then $S$ contains a critical set of $M$.
\end{lem}
\begin{proof}
Suppose $x\in M$ is critical. $S$ is a s-base, thus $x$ is spanned by $S$. So, by \Lref{lem:critical-not-span-by-others}, it must be an element of $S$ (up to projective equivalence).
\end{proof}

In the classical theory (meaning when working with modules over rings), this definition is meaningless; there are no critical vectors. However, In \cite[Theorem 5.24]{Izhaki2010} it is proven that in the supertropical theory, any s-base (if it exists) is a critical set. One could hope that over any entire antiring with a negation map these definition will coincide -- but even in the ELT theory there is a counterexample.

\begin{example}
Consider $ R=\ELT{\mathbb{R}}{\mathbb{C}}$, and
$$M=\Span\underbrace{\left\{ \begin{pmatrix}\layer{0}{1}\\
\layer{0}{1}\\
\layer{0}{0}
\end{pmatrix},\begin{pmatrix}\layer{0}{1}\\
-\infty\\
-\infty
\end{pmatrix},\begin{pmatrix}-\infty\\
\layer{0}{1}\\
\layer{\left(-1\right)}{1}
\end{pmatrix}\right\}}_S$$
The set is a spanning set of $M$. It is straightforward to prove that any vector in $S$ cannot be presented as a linear combination of the others. However,
$$\begin{pmatrix}\layer{0}{1}\\
\layer{0}{1}\\
\layer{0}{0}
\end{pmatrix}=\layer{0}{0}\begin{pmatrix}\layer{0}{1}\\
\layer{0}{1}\\
\layer{0}{0}
\end{pmatrix}+\begin{pmatrix}\layer{0}{1}\\
-\infty\\
-\infty
\end{pmatrix}+\begin{pmatrix}-\infty\\
\layer{0}{1}\\
\layer{\left(-1\right)}{1}
\end{pmatrix}=\begin{pmatrix}\layer{0}{1}\\
\layer{0}{0}\\
\layer{0}{0}
\end{pmatrix}+\begin{pmatrix}-\infty\\
\layer{0}{1}\\
\layer{\left(-1\right)}{1}
\end{pmatrix}$$
implying $\begin{pmatrix}\layer{0}{1}\\
\layer{0}{1}\\
\layer{0}{0}
\end{pmatrix}$ is not critical. Since this vector is not critical, we get an example
of an ELT module with two s-bases which are not projectively equivalent:
$$M=\Span\left\{ \begin{pmatrix}\layer{0}{1}\\
\layer{0}{1}\\
\layer{0}{0}
\end{pmatrix},\begin{pmatrix}\layer{0}{1}\\
-\infty\\
-\infty
\end{pmatrix},\begin{pmatrix}-\infty\\
\layer{0}{1}\\
\layer{\left(-1\right)}{1}
\end{pmatrix}\right\} =\Span\left\{ \begin{pmatrix}\layer{0}{1}\\
\layer{0}{0}\\
\layer{0}{0}
\end{pmatrix},\begin{pmatrix}\layer{0}{1}\\
-\infty\\
-\infty
\end{pmatrix},\begin{pmatrix}-\infty\\
\layer{0}{1}\\
\layer{\left(-1\right)}{1}
\end{pmatrix}\right\}
$$
\end{example}

\subsubsection{d,s-bases}
\begin{defn}\label{def:d,s-base}
Let $R$ be an entire semiring with a negation map, and let $M$ be an $R$-module. A \textbf{d,s-base} of~$M$ is an s-base of $M$, which is also a d-base.
\end{defn}

\begin{lem}\label{lem:size-of-d,s-base}
If $S$ is a d,s-base of $M$, with $M$ finitely generated, then $\left|S\right|=\rk\left(M\right)$.
\end{lem}
\begin{proof}
By \Cref{cor:span-set-at-least-rk-ele}, $\left|S\right|\ge\rk\left(M\right)$; but, since $S$ is also a d-base, we get equality.
\end{proof}

\begin{lem}\label{lem:ind-span-is-d,s}
Any linearly independent spanning set is a d,s-base.
\end{lem}
\begin{proof}
Assume that $S$ is a linearly independent spanning set of $M$. If $S'\subseteq S$ is an s-base, we get that $S'=S$, by \Cref{cor:span-not-sub-ind}.\\

Now, assume that $S\subseteq S''$ is a d-base, using the same corollary we get $S=S''$. So $S$ is a d-base.\\

To sum up, $S$ is an s-base and a d-base, and thus $S$ is a d,s-base.
\end{proof}

Note that not every s-base is a d,s-base.
\begin{example}
Any submodule spanned by zero-layered elements has no d,s-base, since the only linearly independent subset of this submodule is the empty set, which is not a spanning set.
\end{example}

\subsection{Free Modules over Semirings with a Negation Map}

\subsubsection{Definitions and examples}
\begin{defn}\label{def:base}
Let $R$ be an entire semiring with a negation map. An $R$-module $M$ is called \textbf{free}, if there exists a set $B\subseteq M$ such that for all $x\in M\setminus\left\{\zeroM\right\}$ there exists a unique choice of $n\in\mathbb{N}$, $\alpha_{1},\dots,\alpha_{n}\in R$ and~ $x_{1},\dots,x_{n}\in B$ such that
$$x=\alpha_{1}x_{1}+\cdots+\alpha_{n}x_{n}$$
Such a set $B$ is called a \textbf{base} of $M$.
\end{defn}

\begin{example}
Let $R$ be an entire semiring with a negation map. Then $R^{n}$ is a free $R$-module, with the base $\left\{ e_{1},\dots,e_{n}\right\} $. Indeed, the linear combination is determined uniquely, each component separately.
\end{example}

\begin{defn}
Let $R$ be an entire semiring with a negation map, and let $M$ be a free $R$-module with a base~$B$ ($\left|B\right|=n$). For every $x\in M$ there exists a unique representation $x=\alpha_{i_{1}}v_{i_{1}}+\cdots+\alpha_{i_{k}}v_{i_{k}}$. We define the \textbf{coordinate vector} of $x$ by the base $B$, $\left[x\right]_{B}\in R^{n}$, to be:
$$\left(\left[x\right]_{B}\right)_{j}=\begin{cases}
\minf & j\notin\left\{ i_{1},\dots,i_{k}\right\} \\
\alpha_{i_{j}} & j\in\left\{ i_{1},\dots,i_{k}\right\}
\end{cases},\;\; j=1,\dots,n$$
\end{defn}

It is easy to verify that $\left[x+y\right]_{B}=\left[x\right]_{B}+\left[y\right]_{B}$,
and $\left[\alpha x\right]_{B}=\alpha\left[x\right]_{B}$.

\begin{lem}\label{lem:free-isomorphic-Rn}
Let $R$ be an entire semiring with a negation map, and let $M$ be an $R$-module. $M$ is free with a base $B$ of size $n$ if and only if $M\cong R^{n}$.
\end{lem}
\begin{proof}

If $M\cong R^n$, then $M$ is free with a base of size $n$.\\

Now, assume that $M$ is free with a base $B$ of size $n$. Denote $B=\left\{x_{1},\dots,x_{n}\right\} $, and define a function $\varphi:M\rightarrow R^{n}$ by $\varphi\left(x\right)=\left[x\right]_{B}$.
\begin{enumerate}
\item $\varphi$ is an $R$-module homomorphism: since $\left[x+y\right]_{B}=\left[x\right]_{B}+\left[y\right]_{B}$ and $\left[\alpha x\right]_{B}=\alpha\left[x\right]_{B}$.
\item $\varphi$ is injective: If $\varphi\left(x\right)=\varphi\left(y\right)$, then the representations of $x$ and $y$ with respect to $B$ are the same. But we also know that these representations are unique, and thus $x=y$.
\item $\varphi$ is surjective: If $\left(\alpha_{i}\right)\in R^{n}$, then ${\displaystyle \varphi\left(\sum_{i=1}^{n}\alpha_{i}x_{i}\right)=\left(\alpha_{i}\right)}$.
\end{enumerate}
Thus, $M\cong R^{n}$.
\end{proof}

We now prove that any base is a d,s-base.

\begin{lem}\label{lem:lin-comb-zero-base}
If $B$ is a base of $M$, and if
$$\sum_{i=1}^n\alpha_i b_i\in\zeroset{M}$$
for $\alpha_1,\dots,\alpha_n\in R$, $b_1,\dots,b_n\in B$, then $\alpha_1,\dots,\alpha_n\in\zeroset{R}$.
\end{lem}
\begin{proof}
Since $\displaystyle{\sum_{i=1}^n\alpha_i b_i\in\zeroset{M}}$, there is some $x\in M$ such that $\displaystyle{\sum_{i=1}^n\alpha_i b_i=x^\circ}$. Writing $\displaystyle{x=\sum_{b\in B}\beta_b b}$, we have
$$\sum_{i=1}^{n}\alpha_i b_i=x^\circ=\left(\sum_{b\in B}\beta_b b\right)=\sum_{b\in B}\beta_b^\circ b$$
But $B$ is a base, and thus $\alpha_1,\dots,\alpha_n\in\zeroset{R}$ (since all of the coefficients in the right side are quasi-zeros).
\end{proof}

\begin{lem}\label{lem:base-is-lin-ind}
If $B$ is a base of $M$, then it is a linearly independent set.
\end{lem}
\begin{proof}
Assume that
$$\sum_{b\in B}\alpha_{b}b\in\zeroset M$$
By \Lref{lem:lin-comb-zero-base}, $\forall b\in B:\alpha_{b}\in\zeroset{R}$, and thus $B$ is linearly independent.
\end{proof}

\begin{cor}\label{cor:base-is-d,s-base}
Any base is a d,s-base.
\end{cor}
\begin{proof}
Let $M$ be a free $R$-module with base $B$. By definition, $B$ is a spanning set of $M$. In addition, by \Lref{lem:base-is-lin-ind}, $B$ is linearly dependent. We are finished, by \Lref{lem:ind-span-is-d,s}.
\end{proof}

\begin{cor}\label{cor:free-has-base-card-const}
The cardinality of a base of a free module is uniquely determined.
\end{cor}

\begin{example}
We give an example of a submodule of a free module, which has a d,s-base which is not a base. Take $ R=\ELT{\mathbb{R}}{\mathbb{C}}$, and consider
$$M=\Span\left\{ \left(\begin{array}{c}
\layer{1}{1}\\
\layer{0}{1}\\
\layer{0}{1}
\end{array}\right),\left(\begin{array}{c}
\layer{0}{1}\\
\layer{1}{1}\\
\layer{1}{1}
\end{array}\right),\left(\begin{array}{c}
\layer{1}{1}\\
\layer{0}{-1}\\
\layer{0}{1}
\end{array}\right)\right\}$$
The set $\left\{ \left(\begin{array}{c}
\layer{1}{1}\\
\layer{0}{1}\\
\layer{0}{1}
\end{array}\right),\left(\begin{array}{c}
\layer{0}{1}\\
\layer{1}{1}\\
\layer{1}{1}
\end{array}\right),\left(\begin{array}{c}
\layer{1}{1}\\
\layer{0}{-1}\\
\layer{0}{1}
\end{array}\right)\right\} $ is a d,s-base of $M$ (to prove this set is linearly independent, we use \cite[Theorem 3.3.5]{Sheiner2015}; their determinant is $\layer{2}{2}$, which is not zero-layered, and thus this set is linearly dependent). However, it is not a base, because there is a vector which can be written as a linear combination of its elements in two different ways:
$$\left(\begin{array}{c}
\layer{1}{1}\\
\layer{0}{1}\\
\layer{0}{1}
\end{array}\right)+\left(\begin{array}{c}
\layer{0}{1}\\
\layer{1}{1}\\
\layer{1}{1}
\end{array}\right)=\left(\begin{array}{c}
\layer{1}{1}\\
\layer{1}{1}\\
\layer{1}{1}
\end{array}\right)=\left(\begin{array}{c}
\layer{0}{1}\\
\layer{1}{1}\\
\layer{1}{1}
\end{array}\right)+\left(\begin{array}{c}
\layer{1}{1}\\
\layer{0}{-1}\\
\layer{0}{1}
\end{array}\right)$$
\end{example}

However, if we know that a module is free, we also know all of its d,s-bases:

\begin{lem}
Let $M$ be a free $R$-module, and let $S$ be a d,s-base of $M$. Then $S$ is a base of~$M$.
\end{lem}
\begin{proof}
By \Lref{lem:free-isomorphic-Rn}, we can assume $M=R^n$. Since $S$ is an s-base of $M$, by \Lref{lem:s-base-contains-critical}, $S$ contains some critical set, $A$. Since the critical elements in $R^n$ are precisely $\alpha_i e_i$ for some $\alpha_i\in R^{\vee}$, we may write
$$A=\left\{\alpha_1 e_1,\dots,\alpha_n e_n\right\}$$

We now prove that each $\alpha_i$ is invertible. Otherwise, assume without loss of generality that $\alpha_1$ is not invertible. Since $S$ is an s-base of $R^n$, and since $e_1$ is critical, there are $\beta_1 e_1,\dots,\beta_k e_1\in S$ such that
$$\exists\gamma_1,\dots,\gamma_k\in R:\gamma_1\beta_1e_1+\cdots+\gamma_k\beta_ke_1=e_1$$
Since $S$ is linearly independent, we must have $k=1$, and thus $\beta e_1\in S$ such that $\beta$ is invertible. Again, since $S$ is linearly independent, $\alpha_1=\beta$ is invertible, as required.
\end{proof}

\subsubsection{Free modules over entire antirings with a negation map}

For the remainder of this part, we consider free modules over entire antirings with a negation map.

\begin{defn}\label{def:trans-matrix}
Let $R$ be an entire semiring with a negation map, and let $M$ be a free $R$-module with two bases $B=\left\{ v_{1},\dots,v_{n}\right\} $ and $C=\left\{ w_{1},\dots,w_{n}\right\} $. We define the \textbf{transformation matrix} from $B$ to $C$, $\left[I\right]_{C}^{B}$, by:
$$\left[I\right]_{C}^{B}=\left(\left[v_{1}\right]_{C},\dots,\left[v_{n}\right]_{C}\right)$$
Notice that $\left[I\right]_{C}^{B}\in M_n\left(R\right)$.
\end{defn}

\begin{lem}\label{lem:trans-matrix-and-coor-vec}
Under the conditions of \Dref{def:trans-matrix}, $\forall x\in M:\left[I\right]_{C}^{B}\left[x\right]_{B}=\left[x\right]_{C}$.
\end{lem}
\begin{proof}
Denote $\left[I\right]_{C}^{B}=\left(\alpha_{i,j}\right)$, $\left[x\right]_{B}=\left(\beta_{k}\right)$. That means ${\displaystyle x=\sum_{k=1}^{n}\beta_{k}v_{k}}$, and ${\displaystyle v_{k}=\sum_{\ell=1}^{n}\alpha_{\ell,k}w_{\ell}}$. We get
$$x=\sum_{k=1}^{n}\beta_{k}v_{k}=\sum_{k=1}^{n}\beta_{k}\left(\sum_{\ell=1}^{n}\alpha_{\ell,k}w_{\ell}\right)=\sum_{\ell=1}^{n}\left(\sum_{k=1}^{n}\alpha_{\ell,k}\beta_{k}\right)w_{\ell}= \sum_{\ell=1}^{n}\left(\left[I\right]_{C}^{B}\left[x\right]_{B}\right)_{\ell}w_{\ell}$$
And by the definition, $\forall\ell=1,\dots,n:\left(\left[I\right]_{C}^{B}\left[x\right]_{B}\right)_{\ell}=\left(\left[x\right]_{C}\right)_{\ell}$, as needed.
\end{proof}

\begin{lem}\label{lem:trans-matrix-is-invertible}
Under the conditions of \Dref{def:trans-matrix},
$$\left[I\right]_{B}^{C}\cdot\left[I\right]_{C}^{B}=I_{n}$$
\end{lem}
\begin{proof}
We will show equality of columns. For all $i=1,\dots,n$,
$$C_{i}\left(\left[I\right]_{B}^{C}\cdot\left[I\right]_{C}^{B}\right)=\left[I\right]_{B}^{C}\cdot\left[I\right]_{C}^{B}\cdot e_{i}=\left[I\right]_{B}^{C}\cdot\left[I\right]_{C}^{B}\cdot\left[v_{i}\right]_{B}\overset{\Lref{lem:trans-matrix-and-coor-vec}}{=} \left[I\right]_{B}^{C}\cdot\left[v_{i}\right]_{C}\overset{\Lref{lem:trans-matrix-and-coor-vec}}{=}\left[v_{i}\right]_{B}=e_{i}$$
as needed.
\end{proof}

\begin{cor}\label{cor:unique-base}
A base of a free module over an entire antiring with a negation map is unique, up to invertible scalar multiplication of each element and permutation. In other words, any two bases of a free ELT module are projectively equivalent.
\end{cor}
\begin{proof}
If $B$ and $C$ are two bases of a free $R$-module $M$, then $\left[I\right]_{C}^{B}$ is invertible. We also know (by \Cref{cor:free-has-base-card-const}) that $\left[I\right]_{C}^{B}$ is a square matrix. By \cite[Corollary 3]{Dolzan2008}, it is a generalized permutation matrix, namely $C$ is $B$ up to scalar multiplication of each element ($c_{i}$'s) and change of order ($\sigma$).
\end{proof} 
\section{Symmetrized Versions of Important Algebraic Structures}
We now turn to study Lie semialgebras over semirings with a negation map. In the classical theory, a Lie algebra is a vector space endowed with a bilinear alternating multiplication, which satisfies Jacobi's identity. In the context of negation maps, our definition will be quite similar to the classical one. We will mostly follow the approach of \cite{Humpre1972}.





\subsection{Semialgebras over Symmetrized Semirings}

\label{sec:ELT-nonass-algebras}
\begin{defn}
A \textbf{nonassociative semialgebra with a negation map} over a semiring $R$ is an $R$-module with a negation map $A$, together with an $R$-bilinear multiplication, $A\times A\rightarrow A$, that is distributive over addition and satisfies the axioms:
\begin{enumerate}
  \item $\forall\alpha\in R\,\forall x,y\in A:\alpha\left(xy\right)=\left(\alpha x\right)y=x\left(\alpha y\right)$.
  \item $\forall x,y\in A:\minus\left(xy\right)=\left(\minus x\right)y=x\left(\minus y\right)$.
\end{enumerate}
Note that ``nonassociative'' means ``not necessarily associative''.
\end{defn}

As in the case of homomorphisms and congruences, if the negation map on $A$ is induced from a negation map on $R$, the last condition is follows from the other.

\begin{example}
Let $A,B$ be nonassociative semialgebras with a negation map over a semiring $R$. Then $A\times B$ is also a nonassociative semialgebra with a negation map (a module as in \Eref{exa:examples-of-modules}), where the multiplication and the negation map are defined componentwise.
\end{example}

\begin{defn}
An \textbf{associative semialgebra} is a nonassociative semialgebra, whose multiplication is associative.
\end{defn}

\begin{defn}
Let $A,B$ be nonassociative semialgebras with a negation map over a semiring. A function $\varphi:A\rightarrow B$ is called an \textbf{$R$-homomorphism}, if it satisfies the following properties:
\begin{enumerate}
\item $\varphi$ is an $R$-module homomorphism (as defined in \Dref{def:modules-hom}).
\item $\forall x,y\in A:\varphi\left(xy\right)=\varphi\left(x\right)\varphi\left(y\right)$.
\end{enumerate}
The set of all $R$-homomorphisms $\varphi:A\rightarrow B$ is denoted $\Hom{A}{B}$. We also say that:
\begin{itemize}
\item $\varphi$ is an \textbf{$R$-monomorphism}, if $\varphi$ is injective.
\item $\varphi$ is an \textbf{$R$-epimorphism}, if $\varphi$ is surjective.
\item $\varphi$ is an \textbf{$R$-isomorphism}, if $\varphi$ is bijective. In this case we denote $A\cong B$.
\end{itemize}
In addition, if $A$ is a semialgebra with negation over a semiring $R$, then $\End{A}=\Hom{A}{A}$ is also semialgebra with a negation map over $R$, with the usual addition and scalar multiplication, composition as multiplication and negation map
$$\left(\minus\varphi\right)\left(x\right)=\minus\varphi\left(x\right)$$
\end{defn}

One may also define $\asymmdash$-morphism of algebras with a negation map similarly.

\begin{defn}
Let $A,B$ be nonassociative semialgebras with a negation map over a semiring $R$, and let $\varphi:A\rightarrow B$ be an $R$-homomorphism. The \textbf{kernel} of $\varphi$ is
$$\ker\varphi=\left\{x\in A\mid\varphi\left(x\right)\in\zeroset{B}\right\}=\varphi^{-1}\left(\zeroset{B}\right)$$
By \Rref{rem:mod-zeros-sent-zeros}, $\zeroset{A}\subseteq\ker\varphi$.
\end{defn}

We recall the definition of an ideal:
\begin{defn}\label{def:ideal-of-alg}
Let $A$ be a nonassociative semialgebra over a semiring $R$. A subalgebra $I\subseteq A$ which satisfies $IA,AI\subseteq I$ is called an \textbf{ideal} of $A$, denoted $I\ideal A$.
\end{defn}

\begin{lem}
$\ker\varphi$ is an ideal of $A$.
\end{lem}
\begin{proof}
We prove it directly:
\begin{enumerate}
\item If $x,y\in\ker\varphi$, then $\varphi\left(x+y\right)=\varphi\left(x\right)+\varphi\left(y\right)\in\zeroset{B}$. Hence, $x+y\in\ker\varphi$.
\item If $\alpha\in R$ and $x\in\ker\varphi$, then $\varphi\left(\alpha x\right)=\alpha\varphi\left(x\right)\in\zeroset{B}$, and thus $\alpha x\in\ker\varphi$.
\item If $x\in\ker\varphi$, then $\varphi\left(\minus x\right)=\minus\varphi\left(x\right)\in\zeroset{B}$, so $\minus x\in\ker\varphi$.
\item If $x\in\ker\varphi$ and $y\in A$, then $\varphi\left(x\right)\in\zeroset{B}$. Hence,
    $$\varphi\left(xy\right)=\varphi\left(x\right)\varphi\left(y\right)=\left(\zero\varphi\left(x\right)\right)\varphi\left(y\right)= \zero\left(\varphi\left(x\right)\varphi\left(y\right)\right)$$
    Therefore, $\varphi\left(xy\right)\in\zeroset{B}$, implying $xy\in\ker\varphi$.
    Similarly, $yx\in\ker\varphi$.
\end{enumerate}
\end{proof}

\subsection{Lie Semialgebras with a Negation Map}
\subsubsection{Basic definitions}
We first restrict our base semiring to be a \textbf{semifield with a negation map}, i.e.\ a commutative semiring with a negation map in which every element which is not quasi-zero is invertible.

\begin{defn}
Let $R$ be a semifield. A \textbf{Lie semialgebra with a negation map} over $R$ is a nonassociative semialgebra with a negation map over $R$, whose multiplication $\left[\cdot,\cdot\right]:L\times L\rightarrow L$, called a \textbf{negated Lie bracket}, satisfies the following axioms:
\begin{enumerate}
\item\label{itm:bracket-commutes-with-negation} Commutes with the negation map: $\forall x,y\in L:\left[\minus x,y\right]=\minus\left[x,y\right]$.
\item Alternating on $L$: $\forall x\in L:\left[x,x\right]\in\zeroset{L}$.
\item Anticommutativity: $\forall x,y\in L:\left[y,x\right]=\minus\left[x,y\right]$.
\item Jacobi's identity: $\forall x,y,z\in L: \left[x,\left[y,z\right]\right]+\left[z,\left[x,y\right]\right]+\left[y,\left[z,x\right]\right]\in\zeroset{L}$.
\end{enumerate}
\end{defn}

\begin{rem}~
\begin{enumerate}
\item If $\ch R\neq 2$, then $2\Rightarrow1$.
\item If $R$ has a negation map and the negation map of $L$ is the induced negation map from $R$, then condition \ref{itm:bracket-commutes-with-negation} follows from the bilinearity of the negated Lie bracket, since
    $$\left[\minus x,y\right]=\left[\left(\minus\one\right)x,y\right]=\left(\minus\one\right)\left[x,y\right]=\minus\left[x,y\right].$$
\end{enumerate}
\end{rem}

Our definition is a bit different than Rowen's definition \cite{Rowen2016}. The difference is reflected in Jacobi's identity; Rowen's definition requires that
$$\forall x,y,z\in L:\left[x,\left[y,z\right]\right]\minus\left[y,\left[x,z\right]\right]\symmdash\left[\left[x,y\right],z\right]$$
We call this axiom the \textbf{strong Jacobi's identity}. Note that while the strong Jacobi's identity implies our version of Jacobi's identity, the converse does not hold. An example is given in \Eref{exa:disproof-PBW}.

\begin{lem}
$\forall x,y\in L:\left[x,y\right]+\left[y,x\right]\in\zeroset{L}$.
\end{lem}
\begin{proof}
For all $x,y\in L$,
$$\left[x,y\right]+\left[y,x\right]=\left[x,y\right]+\minus\left[x,y\right]\in\zeroset{L}$$
\end{proof}

\begin{defn}
A Lie semialgebra with a negation map $L$ is called \textbf{abelian}, if
$$\forall x,y\in L:\left[x,y\right]\in\zeroset{L}$$
\end{defn}

\begin{example}\label{exa:comm-is-Lie}
Given a semifield $R$, let $A$ be an associative semialgebra with a negation map over $R$. We define an operation as follows: $\left[x,y\right]=xy+\minus yx$, which is called the \textbf{negated commutator}. We will now show that it forms a structure of a Lie semialgebra with a negation map over $R$, which satisfied the strong Jacobi's identity:
\begin{enumerate}
\item Bilinear: Let $x,y,z\in A$ and $\alpha,\beta\in R$. Then,
$$\begin{array}{c}
\left[\alpha x+\beta y,z\right]=\left(\alpha x+\beta y\right)z+\minus\left(z\left(\alpha x+\beta y\right)\right)=\\
=\alpha\left(xz+\minus zx\right)+\beta\left(yz+\minus zy\right)=\alpha\left[x,z\right]+\beta\left[y,z\right]
\end{array}$$
Linearity in the second component is proved similarly.
\item $\left[\cdot,\cdot\right]$ commutes with $\minus$: For all $x,y\in L$,
$$\left[\minus x,y\right]=\left(\minus x\right)y+\minus\left(y\left(\minus x\right)\right)=\minus\left(xy+\minus yx\right)=\minus\left[x,y\right]$$
\item Alternating on $A$: For all $x\in A$,
$$\left[x,x\right]=x^{2}+\minus x^{2}\in\zeroset{L}$$
\item Anticommutativity: For all $x,y\in A$,
$$\left[y,x\right]=yx+\minus xy=\minus\left(xy+\minus yx\right)=\minus\left[x,y\right]$$
\item Strong Jacobi's identity: a proof can be found in \cite{Rowen2016}, which uses the strong transfer principle.
\end{enumerate}
\end{example}

\begin{lem}\label{lem:brac-zero-is-zero}
$$\forall x,y\in L:x\in\zeroset{L}\vee y\in\zeroset{L}\Longrightarrow\left[x,y\right]\in\zeroset{L}$$
\end{lem}
\begin{proof}
Assume $x\in\zeroset{L}$; write $x=z+\minus z$ for some $z\in L$. Then
$$\left[x,y\right]=\left[z+\minus z,y\right]=\left[z,y\right]+\minus\left[z,y\right]\in\zeroset{L}$$
\end{proof}

\begin{defn}
Let $R$ be a semifield with a negation map, and let $L$ be a Lie semialgebra with a negation map over $R$. A subset $L_1\subseteq L$ is called a \textbf{subalgebra} of $L$, if it is a Lie semialgebra with a negation map over $R$ with the restrictions of the operations of $L$.
\end{defn}

\begin{defn}\label{def:center-of-an-ELT-Lie-alg}
Let $R$ be a semifield with a negation map, and let $L$ be a Lie semialgebra with a negation map over $R$. The \textbf{center} of~$L$ is
$$Z\left(L\right)=\left\{ x\in L\mid\forall y\in L:\left[x,y\right]\in\zeroset{L}\right\}$$
\end{defn}

\begin{lem}\label{lem:zero-layered-in-center}
$\zeroset{L}\subseteq Z\left(L\right)$.
\end{lem}
\begin{proof}
By \Lref{lem:brac-zero-is-zero}.
\end{proof}

\begin{defn}
Let $R$ be a semifield with a negation map, let $L$ be a Lie semialgebra with a negation map over $R$, and let $L_1,L_2\subseteq L$ be subalgebras of $L$. Define
$$\left[L_1,L_2\right]=\Span\left\{\left[x_1,x_2\right]\middle|x_1\in L_1, x_2\in L_2\right\}$$
\end{defn}

\begin{lem}\label{lem:center-is-ideal}
$$\left[Z\left(L\right),L\right]\subseteq\zeroset{L}\subseteq Z\left(L\right)$$
\end{lem}
\begin{proof}
$$\forall x\in Z\left(L\right)\,\forall y\in L:\left[x,y\right]\in\zeroset{L}\subseteq Z\left(L\right)$$
\end{proof}

\subsubsection{Homomorphisms and Ideals}

We use the general definitions for homomorphisms and ideals of nonassociative algebras given in \Sref{sec:ELT-nonass-algebras}.
\begin{lem}
$\zeroset{L}\ideal L$ and $Z\left(L\right)\ideal L$.
\end{lem}
\begin{proof}
$\zeroset{L}$ is immediate by \Lref{lem:brac-zero-is-zero}, and $Z\left(L\right)$ by \Lref{lem:center-is-ideal}.
\end{proof}

\begin{lem}
If $I,J\ideal L$, then $I\cap J\ideal L$.
\end{lem}
\begin{proof}
Let $x\in I\cap J$, $y\in L$. $x\in I$, so $\left[x,y\right]\in I$. Also $x\in J$, so $\left[x,y\right]\in J$. In conclusion, $\left[x,y\right]\in I\cap J$.
\end{proof}

\begin{defn}
Let $R$ be a semifield, let $L$ be a Lie semialgebra with a negation map over $R$, and let $I,J\ideal L$ be two ideals of $L$. Then their \textbf{sum} is defined as
$$I+J=\left\{x+y\middle| x\in I,y\in J\right\}$$
\end{defn}

\begin{lem}
If $I,J\ideal L$, then $I+J\ideal L$.
\end{lem}
\begin{proof}
Obviously, $I+J$ is a submodule of $L$. Now, let $x\in I+J$, $y\in L$. $x\in I+J$, so there exists $x'\in I$ and $x''\in J$ such that $x=x'+x''$. We get
$$\left[x,y\right]=\left[x'+x'',y\right]=\left[x',y\right]+\left[x'',y\right]\in I+J$$
\end{proof}

\subsubsection{The Adjoint Algebra}

\begin{defn}
Let $R$ be a semifield, let $L$ be a Lie semialgebra with a negation map over $R$, and let $x\in L$. We define a homomorphism $\ad_x:L\to L$ by
$$\ad_x\left(y\right)=\left[x,y\right]$$
\end{defn}

\begin{defn}
Let $R$ be a semifield, and let $L$ be a Lie semialgebra with a negation map over~$R$. The \textbf{adjoint algebra} of $L$ is the set
$$\Ad L=\left\{\ad_{x}\middle| x\in L\right\}$$
endowed with the following negated Lie bracket:
$$\forall\ad_{x},\ad_{y}\in\Ad L:\left[\ad_{x},\ad_{y}\right]=\ad_{\left[x,y\right]}$$
and the induced negation map from $\End{L}$: $\minus\ad_x=\ad_{\minus x}$.
\end{defn}

Jacobi's identity can be rewritten now as
$$\forall x,y,z\in L:\ad_x\left(\ad_y\left(z\right)\right)+\ad_z\left(\ad_x\left(y\right)\right)+\ad_y\left(\ad_z\left(x\right)\right)\in\zeroset{L}$$
and the strong Jacobi's identity can be rewritten now as
$$\ad_x\ad_y+\minus\ad_y\ad_x\symmdash\ad_{\left[x,y\right]}$$

\begin{lem}
$\Ad L$ is a Lie semialgebra with a negation map over $R$.
\end{lem}
\begin{proof}
We first check that $\Ad L$ is a submodule of $\End{L}$. But this is obvious, since $\ad_x+\ad_y=\ad_{x+y}\in\Ad L$, $\alpha\ad_x=\ad_{\alpha x}\in\Ad L$ and $\minus\ad_x=\ad_{\minus x}\in\Ad L$.\\

Now, we need to check that it is a Lie semialgebra with a negation map.
\begin{enumerate}
\item $\forall\alpha,\beta\in R\;\forall\ad_x,\ad_y,\ad_z\in\Ad L:\left[\alpha\ad_x+\beta\ad_y,\ad_z\right]=\left[\ad_{\alpha x+\beta y},\ad_z\right]=\ad_{\left[\alpha x+\beta y,z\right]}=\ad_{\alpha\left[x,z\right]+\beta\left[y,z\right]}=\alpha\ad_{\left[x,z\right]}+\beta\ad_{\left[y,z\right]}=\alpha\left[\ad_x,\ad_z\right]+\beta\left[\ad_y,\ad_z\right]$
\item $\forall\ad_x,\ad_y\in\Ad L:\left[\minus\ad_x,\ad_y\right]=\left[\ad_{\minus x},\ad_y\right]=\ad_{\left[\minus x,y\right]}=\minus\ad_{\left[x,y\right]}=\minus\left[\ad_x,\ad_y\right]$.
\item $\forall\ad_x\in\Ad L:\left[\ad_x,\ad_x\right]=\ad_{\left[x,x\right]}$, which is a quasi-zero.
\item $\forall\ad_x,\ad_y\in\Ad L:\left[\ad_y,\ad_x\right]=\ad_{\left[y,x\right]}=\ad_{\minus\left[x,y\right]}=\minus\ad_{\left[x,y\right]}=\minus\left[\ad_x,\ad_y\right]$.
\item $\forall\ad_x,\ad_y,\ad_z\in\Ad L:\left[\ad_x\left[\ad_y,\ad_z\right]\right]+\left[\ad_z\left[\ad_x,\ad_y\right]\right]+\left[\ad_y\left[\ad_z,\ad_x\right]\right]=\ad_{\left[x\left[y,z\right]\right]+\left[z\left[x,y\right]\right]+\left[y\left[z,x\right]\right]}$, which is a quasi-zero.
\end{enumerate}
\end{proof}

\begin{lem}
There is a homomorphism of Lie semialgebras with a negation map $L\to\Ad L$ given by
$$x\mapsto\ad_x$$
Therefore, given the congruence $\equiv$ on $L$ defined by $x\equiv y\Leftrightarrow\ad_{x}=\ad_{y}$, one has
$$L/\equiv\;\cong\Ad L$$
\end{lem}
\begin{proof}
Using the definition of the operation on $\Ad L$, this is a homomorphism. By \Tref{thm:first-iso-thm}, we get the conclusion.
\end{proof}

A word of caution: with the definition given above to the negated Lie bracket, $\Ad L$ is not a subalgebra of $\End{L}$. Generally, there is no obvious way of fixing this problem; however, one can use the strong Jacobi's identity.\\

If $L$ satisfies the strong Jacobi's identity, then one can define
$$\ad L=\overline{\left\{\ad_x\mid x\in L\right\}}\subseteq\End{L}$$
as in \cite{Rowen2016}, and this is also a Lie semialgebra with a negation map over $R$ (under the negated commutator). Nevertheless, we now have only a $\asymmdash$-morphism $L\to\ad L$ given by $x\mapsto\ad_x$.

\subsection{Free Lie Algebras with a Negation Map}

In this subsection we turn to study Lie semialgebras with a negation map which are free as modules and have a base consisting of one, two or three elements.

\subsubsection{$1$-dimensional Lie Algebras with a negation map}

\begin{lem}
Any Lie semialgebra $L$ over a semifield with a negation map $R$ with a base $B=\left\{x\right\}$ is abelian.
\end{lem}
\begin{proof}
Given $y,z\in L$, there exist $\alpha,\beta\in R$ such that $y=\alpha x,z=\beta x$. Therefore,
$$\left[y,z\right]=\left[\alpha x,\beta x\right]=\alpha\beta\left[x,x\right]\in\zeroset{L}$$
Thus, $L$ is abelian.
\end{proof}

\subsubsection{$2$-dimensional Lie Algebras with a negation map}

\begin{lem}\label{lem:2-dim-free-ELT-Lie-alg}
Let $R$ be a semifield with a negation map, and let $L$ be a free $R$-module with base $B=\left\{x_1,x_2\right\}$. Define
$$\left[x_1,x_1\right]=\alpha_1 x_1+\alpha_2 x_2,\;\;\;\left[x_2,x_2\right]=\beta_1 x_1+\beta_2 x_2,\;\;\;\left[x_1,x_2\right]=\minus\left[x_2,x_1\right]=\gamma_1 x_1+\gamma_2 x_2$$
where $\alpha_1,\alpha_2,\beta_1,\beta_2,\gamma_1,\gamma_2\in R$ and $\alpha_1,\alpha_2,\beta_1,\beta_2\in\zeroset{R}$. Now, extend $\left[\cdot,\cdot\right]$ to~$\left[\cdot,\cdot\right]:L\times L\to L$ by bilinearity.

Then $L$ equipped with $\left[\cdot,\cdot\right]$ is a Lie semialgebra over $R$.
\end{lem}
\begin{proof}
By bilinearity, it is enough to check Jacobi's identity in two cases:
\begin{enumerate}
\item $\left[x_1,\left[x_1,x_2\right]\right]+\left[x_1,\left[x_2,x_1\right]\right]+\left[x_2,\left[x_1,x_1\right]\right]\in\zeroset{L}$.
\item $\left[x_2,\left[x_2,x_1\right]\right]+\left[x_2,\left[x_1,x_2\right]\right]+\left[x_1,\left[x_2,x_2\right]\right]\in\zeroset{L}$.
\end{enumerate}

We shall prove the first part; the second is proved similarly. Indeed,
$$\left[x_1,\left[x_1,x_2\right]\right]+\left[x_1,\left[x_2,x_1\right]\right]+\left[x_2,\left[x_1,x_1\right]\right]= \left[x_1,\left[x_1,x_2\right]^\circ\right]+\left[x_2,\left[x_1,x_1\right]\right]=\left[x_1,\left[x_1,x_2\right]\right]^\circ+\left[x_2,\left[x_1,x_1\right]\right]$$
We note that $\left[x_1,x_1\right]\in\zeroset{L}$, and thus $\left[x_1,\left[x_1,x_2\right]\right]+\left[x_1,\left[x_2,x_1\right]\right]+\left[x_2,\left[x_1,x_1\right]\right]\in\zeroset{L}$, as required.
\end{proof}

\begin{lem}
Let $L$ be a Lie semialgebra over a semifield with a negation map $R$ with base $B=\left\{x_1,x_2\right\}$. Then there are~$\alpha_1,\alpha_2,\beta_1,\beta_2,\gamma_1,\gamma_2\in R$ such that
$$\left[x_1,x_1\right]=\alpha_1 x_1+\alpha_2 x_2,\;\;\;\left[x_2,x_2\right]=\beta_1 x_1+\beta_2 x_2,\;\;\;\left[x_1,x_2\right]=\minus\left[x_2,x_1\right]=\gamma_1 x_1+\gamma_2 x_2$$
and $\alpha_1,\alpha_2,\beta_1,\beta_2\in\zeroset{R}$. Therefore, each such Lie semialgebra is obtained from \Lref{lem:2-dim-free-ELT-Lie-alg}.
\end{lem}
\begin{proof}
The existence of $\alpha_1,\alpha_2,\beta_1,\beta_2,\gamma_1,\gamma_2\in R$ follows from the fact that $L$ is free, and by the antisymmetry of $\left[\cdot,\cdot\right]$. Since $\left[\cdot,\cdot\right]$ is alternating, combined with \Lref{lem:lin-comb-zero-base}, $\alpha_1,\alpha_2,\beta_1,\beta_2\in\zeroset{R}$.
\end{proof}

\subsubsection{$3$-dimensional Lie Algebras with a negation map}

We turn to studying free Lie semialgebras with a base consisting of three elements. The purpose of this case is to define a Lie semialgebra that will be parallel to $\sl\left(2,F\right)$ in the classical theory.

We are now going to give a necessary and sufficient condition for a bilinear multiplication defined on a module over a semifield with a negation map to be a negated Lie bracket. This is formulated in \Cref{cor:3-dim-free-ELT-Lie-alg} and in \Lref{lem:3-dim-free-ELT-Lie-alg-is-all}.

Throughout this part, $R$ is a semifield with a negation map.

\begin{lem}\label{lem:3-dim-free-ELT-Lie-alg}
Let $L$ be a free $R$-module with base $B=\left\{x_1,x_2,x_3\right\}$. Define
$$\forall i\leq j:\left[x_i,x_j\right]=\minus\left[x_j,x_i\right]=\sum_{\ell=1}^3\alpha_{i,j,\ell}x_\ell$$
where $\forall i,j,\ell:\alpha_{i,j,\ell}\in R$, and $\forall i,\ell:\alpha_{i,i,\ell}\in\zeroset{R}$. If $i>j$, we write $\alpha_{i,j,\ell}=\minus\alpha_{j,i,\ell}$. Assume also that
$$\forall i,j,k,m:\sum_{i,j,k}\sum_{\ell=1}^3\alpha_{i,\ell,m}\alpha_{j,k,\ell}\in\zeroset{R}$$
where $\displaystyle{\sum_{i,j,k}}$ is the cyclic sum over $i,j,k$. Extend $\left[\cdot,\cdot\right]$ to $\left[\cdot,\cdot\right]:L\times L\to L$ by bilinearity.

Then $L$ equipped with $\left[\cdot,\cdot\right]$ is a Lie semialgebra over $R$.
\end{lem}
\begin{proof}
Again, it is enough to check Jacobi's identity. We need to ensure that
$$\forall i,j,k:\left[x_i,\left[x_j,x_k\right]\right]+\left[x_k,\left[x_i,x_j\right]\right]+\left[x_j,\left[x_k,x_i\right]\right]\in\zeroset{L}$$

Take some $i,j,k$. By calculating $\left[x_i,\left[x_j,x_k\right]\right]$, we get:
$$\left[x_i,\left[x_j,x_k\right]\right]=\left[x_i,\sum_{\ell=1}^3\alpha_{j,k,\ell}x_\ell\right]=\sum_{\ell=1}^3\alpha_{j,k,\ell}\left[x_i,x_\ell\right]= \sum_{\ell=1}^3\left(\alpha_{j,k,\ell}\sum_{m=1}^3\alpha_{i,\ell,m}x_m\right)=\sum_{m=1}^3\left(\sum_{\ell=1}^3\alpha_{j,k,\ell}\alpha_{i,\ell,m}\right)x_m$$
By permuting the indices,
\begin{eqnarray*}
\left[x_k,\left[x_i,x_j\right]\right]&=&\sum_{m=1}^3\left(\sum_{\ell=1}^3\alpha_{i,j,\ell}\alpha_{k,\ell,m}\right)x_m\\
\left[x_j,\left[x_k,x_i\right]\right]&=&\sum_{m=1}^3\left(\sum_{\ell=1}^3\alpha_{k,i,\ell}\alpha_{j,\ell,m}\right)x_m
\end{eqnarray*}
By summing all of the above, we get
$$\left[x_i,\left[x_j,x_k\right]\right]+\left[x_k,\left[x_i,x_j\right]\right]+\left[x_j,\left[x_k,x_i\right]\right]= \sum_{m=1}^3\left(\sum_{\textnormal{cyc}}\sum_{\ell=1}^3\alpha_{j,k,\ell}\alpha_{i,\ell,m}\right)x_m$$

Thus, the condition
$$\forall m:\sum_{i,j,k}\sum_{\ell=1}^3\alpha_{i,\ell,m}\alpha_{j,k,\ell}\in\zeroset{R}$$
is equivalent to Jacobi's identity (by \Lref{lem:lin-comb-zero-base}), and the assertion follows.
\end{proof}

The above condition involving the cyclic sum may seem rather strong, since it has to hold for each choice of $1\leq i,j,k,m\leq 3$. However, it
turns out that it is enough to check that it holds for a specific choice of $i,j,k$ in which they are all different -- for example,~$\left(i,j,k\right)=\left(1,2,3\right)$.
This is formulated in the following lemma:

\begin{lem}\label{lem:3-dim-free-ELT-Lie-alg-simplify}
Assume that $\forall i,\ell:\alpha_{i,i,\ell}\in\zeroset{R}$ and $\forall i,j,\ell:\alpha_{i,j,\ell}=\minus\alpha_{j,i,\ell}$. If
$$\forall m:\sum_{\ell=1}^3\left(\alpha_{1,\ell,m}\alpha_{2,3,\ell}+\alpha_{3,\ell,m}\alpha_{1,2,\ell}+\alpha_{2,\ell,m}\alpha_{3,1,\ell}\right)\in\zeroset{R}$$
holds, then
$$\forall i,j,k,m:\sum_{i,j,k}\sum_{\ell=1}^3\alpha_{i,\ell,m}\alpha_{j,k,\ell}\in\zeroset{R}$$
\end{lem}
\begin{proof}
We prove the lemma in two cases.
\begin{casenv}
\item We first assume that two of $i,j,k$ are equal. Without loss of generality, we assume that $j=k$, and thus
    $\forall\ell:\alpha_{j,k,\ell}\in\zeroset{R}$, and $\forall\ell:\alpha_{i,j,\ell}=\minus\alpha_{k,i,\ell}$. We expand the cyclic sum:
    \begin{eqnarray*}
    \sum_{i,j,k}\sum_{\ell=1}^3\alpha_{i,\ell,m}\alpha_{j,k,\ell}&=&\sum_{\ell=1}^3\left(\alpha_{i,\ell,m}\alpha_{j,k,\ell}+
    \alpha_{k,\ell,m}\alpha_{i,j,\ell}+\alpha_{j,\ell,m}\alpha_{k,i,\ell}\right)=\\
    &=&\sum_{\ell=1}^3\alpha_{i,\ell,m}\alpha_{j,k,\ell}+\sum_{\ell=1}^3\left(\alpha_{k,\ell,m}\alpha_{i,j,\ell}+
    \minus\alpha_{j,\ell,m}\alpha_{i,j,\ell}\right)=\\
    &=&\sum_{\ell=1}^3\zero\alpha_{i,\ell,m}\alpha_{j,k,\ell}+\sum_{\ell=1}^3\zero\alpha_{k,\ell,m}\alpha_{i,j,\ell}=\\
    &=&\zero\left(\sum_{\ell=1}^3\alpha_{i,\ell,m}\alpha_{j,k,\ell}+\sum_{\ell=1}^3\alpha_{k,\ell,m}\alpha_{i,j,\ell}\right)
    \end{eqnarray*}
    and thus the sum is quasi-zero.\\

\item Now we may assume that $i,j,k$ are different indices. Since $1\leq i,j,k\leq 3$, and by the cyclic symmetry of $i,j,k$, there are two options:
    either $\left(i,j,k\right)=\left(1,2,3\right)$, and the sum is
    $$S_{\left(1,2,3\right)}=\sum_{i,j,k}\sum_{\ell=1}^3\alpha_{i,\ell,m}\alpha_{j,k,\ell}=\sum_{\ell=1}^3\left(\alpha_{1,\ell,m}\alpha_{2,3,\ell}+
    \alpha_{3,\ell,m}\alpha_{1,2,\ell}+\alpha_{2,\ell,m}\alpha_{3,1,\ell}\right)$$
    or $\left(i,j,k\right)=\left(1,3,2\right)$, and the sum is
    $$S_{\left(1,3,2\right)}=\sum_{i,j,k}\sum_{\ell=1}^3\alpha_{i,\ell,m}\alpha_{j,k,\ell}=\sum_{\ell=1}^3\left(\alpha_{1,\ell,m}\alpha_{3,2,\ell}+
    \alpha_{2,\ell,m}\alpha_{1,3,\ell}+\alpha_{3,\ell,m}\alpha_{2,1,\ell}\right)$$
    We notice that $S_{\left(1,3,2\right)}=\minus S_{\left(1,2,3\right)}$, since
    \begin{eqnarray*}
    S_{\left(1,3,2\right)}&=&\sum_{\ell=1}^3\left(\alpha_{1,\ell,m}\alpha_{3,2,\ell}+\alpha_{2,\ell,m}\alpha_{1,3,\ell}+
    \alpha_{3,\ell,m}\alpha_{2,1,\ell}\right)=\\
    &=&\sum_{\ell=1}^3\left(\minus\alpha_{1,\ell,m}\alpha_{2,3,\ell}+\minus\alpha_{2,\ell,m}\alpha_{3,1,\ell}+
    \minus\alpha_{3,\ell,m}\alpha_{1,2,\ell}\right)=\\
    &=&\minus\sum_{\ell=1}^3\left(\alpha_{1,\ell,m}\alpha_{2,3,\ell}+\alpha_{2,\ell,m}\alpha_{3,1,\ell}+\alpha_{3,\ell,m}\alpha_{1,2,\ell}\right)=
    \minus S_{\left(1,2,3\right)}
    \end{eqnarray*}
    But we assumed that $S_{\left(1,2,3\right)}\in\zeroset{R}$, and thus $S_{\left(1,3,2\right)}\in\zeroset{R}$, as required.
\end{casenv}
\end{proof}

\begin{cor}\label{cor:3-dim-free-ELT-Lie-alg}
Let $L$ be a free $R$-module with base $B=\left\{x_1,x_2,x_3\right\}$. Define
$$\forall i\leq j:\left[x_i,x_j\right]=\minus\left[x_j,x_i\right]=\sum_{\ell=1}^3\alpha_{i,j,\ell}x_\ell$$
where $\forall i,j,\ell:\alpha_{i,j,\ell}\in R$, and $\forall i,\ell:\alpha_{i,i,\ell}\in\zeroset{R}$. If $i>j$, we write $\alpha_{i,j,\ell}=\minus\alpha_{j,i,\ell}$. Assume that also
$$\forall m:\sum_{\ell=1}^3\left(\alpha_{1,\ell,m}\alpha_{2,3,\ell}+\alpha_{3,\ell,m}\alpha_{1,2,\ell}+\alpha_{2,\ell,m}\alpha_{3,1,\ell}\right)\in\zeroset{R}$$
Extend $\left[\cdot,\cdot\right]$ to $\left[\cdot,\cdot\right]:L\times L\to L$ by bilinearity.

Then $L$ equipped with $\left[\cdot,\cdot\right]$ is a Lie semialgebra over $R$.
\end{cor}
\begin{proof}
By \Lref{lem:3-dim-free-ELT-Lie-alg-simplify}, the conditions of \Lref{lem:3-dim-free-ELT-Lie-alg} hold, and thus the assertion follows.
\end{proof}

\begin{lem}\label{lem:3-dim-free-ELT-Lie-alg-is-all}
Let $L$ be a Lie semialgebra over $R$ with base $B=\left\{x_1,x_2,x_3\right\}$. Then there are~$\alpha_{i,j,\ell}\in R$ such that
$$\forall i\leq j:\left[x_i,x_j\right]=\minus\left[x_j,x_i\right]=\sum_{\ell=1}^3\alpha_{i,j,\ell}x_\ell$$
and $\forall i,j,\ell:\alpha_{i,j,\ell}\in R$, and $\forall i,\ell:\alpha_{i,i,\ell}\in\zeroset{R}$. If $i>j$, we denote $\alpha_{i,j,\ell}=\minus\alpha_{j,i,\ell}$. Moreover,
$$\forall m:\sum_{\ell=1}^3\left(\alpha_{1,\ell,m}\alpha_{2,3,\ell}+\alpha_{3,\ell,m}\alpha_{1,2,\ell}+\alpha_{2,\ell,m}\alpha_{3,1,\ell}\right)\in\zeroset{R}$$
Therefore, each such Lie semialgebra is obtained from \Cref{cor:3-dim-free-ELT-Lie-alg}.
\end{lem}
\begin{proof}
The existence of $\alpha_{i,j,\ell}$ follows from the fact that $L$ is free, and by the antisymmetry of $\left[\cdot,\cdot\right]$. Since $\left[\cdot,\cdot\right]$ is also alternating, and by \Lref{lem:lin-comb-zero-base}, $\forall i,\ell:\alpha_{i,i,\ell}\in\zeroset{R}$.\\

In addition, in the proof of \Lref{lem:3-dim-free-ELT-Lie-alg}, we showed that Jacobi's identity is equivalent to
$$\forall i,j,k,m:\sum_{i,j,k}\sum_{\ell=1}^3\alpha_{j,k,\ell}\alpha_{i,\ell,m}\in\zeroset{R}$$
In particular, substituting $\left(i,j,k\right)=\left(1,2,3\right)$,
$$\forall m:\sum_{\ell=1}^3\left(\alpha_{1,\ell,m}\alpha_{2,3,\ell}+\alpha_{3,\ell,m}\alpha_{1,2,\ell}+\alpha_{2,\ell,m}\alpha_{3,1,\ell}\right)\in\zeroset{R}$$
and thus we are finished.
\end{proof}

We use the above way to construct a Lie semialgebra which has the same relations as $\sl\left(2,F\right)$. In \Sref{sec:classical-ELT-Lie-alg} we will see the naive way of constructing $\sl\left(n,R\right)$, which will be different from the following Lie semialgebra:

\begin{example}\label{exa:sl2-type}
We take a free $R$-module $L$ with base $\left\{e,f,h\right\}$, and define
\begin{eqnarray*}
\left[e,e\right]=\left[f,f\right]=\left[h,h\right]&=&0_L\\
\left[e,f\right]=\minus\left[f,e\right]&=&h\\
\left[h,e\right]=\minus\left[e,h\right]&=&2e\\
\left[h,f\right]=\minus\left[f,h\right]&=&2f
\end{eqnarray*}
where $2e$ and $2f$ mean $e+e$ and $f+f$. In the notations of \Lref{lem:3-dim-free-ELT-Lie-alg}, if $x_1=e$, $x_2=f$ and $x_3=h$, then
$$\alpha_{1,2,3}=\minus\alpha_{2,1,3}=\one,\;\;\;\alpha_{3,1,1}=\minus\alpha_{1,3,1}=2,\;\;\;\alpha_{3,2,2}=\minus\alpha_{2,3,2}=\minus 2$$
and all other coefficients are $\minf$. We want to prove that these coefficients satisfy the conditions of \Cref{cor:3-dim-free-ELT-Lie-alg}, and thus $L$ equipped with the bilinear extension of $\left[\cdot,\cdot\right]$ is a Lie semialgebra over $R$.

We need to ensure that
$$\forall m:\sum_{\ell=1}^3\left(\alpha_{1,\ell,m}\alpha_{2,3,\ell}+\alpha_{3,\ell,m}\alpha_{1,2,\ell}+\alpha_{2,\ell,m}\alpha_{3,1,\ell}\right)\in\zeroset{R}$$

\begin{enumerate}
\item If $m=1$, $\alpha_{i,\ell,1}=\minf$ unless $\left(i,\ell\right)\in\left\{\left(1,3\right),\left(3,1\right)\right\}$. We are left with two
    summands $\neq\minf$:
    $$\alpha_{1,3,1}\alpha_{2,3,3}+\alpha_{3,1,1}\alpha_{1,2,1}=\minf$$
\item If $m=2$, $\alpha_{i,\ell,2}=\minf$ unless $\left(i,\ell\right)\in\left\{\left(2,3\right),\left(3,2\right)\right\}$. We are left with two
    summands $\neq\minf$:
    $$\alpha_{2,3,2}\alpha_{3,1,3}+\alpha_{3,2,2}\alpha_{1,2,2}=\minf$$
\item If $m=3$, $\alpha_{i,\ell,3}=\minf$ unless $\left(i,\ell\right)\in\left\{\left(1,2\right),\left(2,1\right)\right\}$. We are left with two
    summands $\neq\minf$:
    $$\alpha_{1,2,3}\alpha_{2,3,2}+\alpha_{2,1,3}\alpha_{3,1,1}=1\cdot 2\minus 1\cdot 2=2^\circ$$
\end{enumerate}

Since the sum is quasi-zero for each choice of $m$, such $L$ exists.
\end{example}

\subsection{Symmetrized Versions of the Classical Lie Algebras}\label{sec:classical-ELT-Lie-alg}
In this subsection we construct negated versions of the classical Lie algebras: $A_{n}$, $B_{n}$, $C_{n}$ and~$D_{n}$. We assume again that our underlying semiring $R$ is a semifield with a negation map.

\begin{defn}
The \textbf{general linear algebra}, $\gl\left(n, R\right)$, is the Lie semialgebra of all matrices of size $n\times n$, where the Lie bracket is the negated commutator.
\end{defn}

\begin{defn}
We define matrices $e_{i,j}\in\gl\left(n, R\right)$ by
$$\left(e_{i,j}\right)_{k,\ell}=\begin{cases}
\one & \left(i,j\right)=\left(k,\ell\right)\\
\minf & \left(i,j\right)\neq\left(k,\ell\right)
\end{cases}$$
these matrices form a base of $\gl\left(n, R\right)$.
\end{defn}

\begin{rem}\label{elem-com-formula}
In $\gl\left(n, R\right)$, we have the formula
$$\left[e_{i,j},e_{k,\ell}\right]=\delta_{j,k}e_{i,\ell}+\minus\delta_{i,\ell}e_{k,j}$$
where
$$\delta_{i,j}=\begin{cases}
\one & i=j\\
\minf & i\neq j
\end{cases}$$
\end{rem}

\begin{defn}
$A_{n}$, or the \textbf{negated special linear algebra}, is
$$A_{n}=\left\{ A\in\gl\left(n+1, R\right)\mid s\left(\tr\left(A\right)\right)=0\right\}$$
\end{defn}

\begin{lem}
$A_{n}$ is a subalgebra of $\gl\left(n+1, R\right)$.
\end{lem}
\begin{proof}
Obviously, $A_{n}$ is a submodule of $\gl\left(n+1, R\right)$. Since $\tr\left(AB\right)=\tr\left(BA\right)$, we have
$$s\left(\tr\left(\left[A,B\right]\right)\right)=s\left(\tr\left(AB+\minus BA\right)\right)=0$$
so $A_{n}$ is a subalgebra of $\gl\left(n+1, R\right)$.
\end{proof}

\begin{lem}
$A_{n}$ has the following s-base:
$$B=\left\{ e_{i,j}\mid i\neq j\right\} \cup\left\{ e_{i,i}+\minus e_{j,j}\mid1\leq i<j\leq n\right\}\cup\left\{\zero e_{i,i}\mid 1\leq i\leq n\right\}$$
\end{lem}
\begin{proof}
Assume that $A=\left(a_{i,j}\right)\in A_{n}$. Therefore, $s\left(\tr\left(A\right)\right)=0$. There are two options:
\begin{enumerate}
\item If $\tr\left(A\right)$ is achieved from a dominant zero-layered element in the diagonal of $A$, without loss of generality $a_{1,1}$, then
    $$A=\sum_{\substack{i,j=1\\i\neq j}}^{n}a_{i,j}e_{i,j}+\sum_{i=1}^n a_{i,i}\left(e_{i,i}+\minus e_{1,1}\right)$$
\item Otherwise, $\tr\left(A\right)$ is achieved from (at least) two elements, not of layer zero, such that they ``cancel'' each other. Assume without loss of generality that these elements are $a_{1,1},\dots,a_{k,k}$; that means that $\t\left(a_{1,1}\right)=\cdots=t\left(a_{k,k}\right)=\t\left(\tr\left(A\right)\right)$, $s\left(a_{1,1}\right),\dots,s\left(a_{k,k}\right)\neq 0$ and for each~$\ell>k$, either $\t\left(a_{\ell,\ell}\right)<\t\left(\tr\left(A\right)\right)$ or $s\left(a_{\ell,\ell}\right)=0$. Then
    $$A=\sum_{\substack{i,j=1\\i\neq j}}^{n}a_{i,j}e_{i,j}+\sum_{i=2}^n a_{i,i}\left(e_{i,i}+\minus e_{1,1}\right)$$
\end{enumerate}
Therefore, $B$ spans $A_n$. It is easy to see that no element of $B$ is spanned by the others, and thus it is an s-base.
\end{proof}

\begin{lem}
Considering $A_{1}$, we see that its s-base consists of the elements
$$e=e_{1,2},\quad f=e_{2,1},\quad h=e_{1,1}+\minus e_{2,2}$$
they satisfy the relations
$$\left[e,f\right]=h;\quad\left[e,h\right]=\layer{0}{-2}e;\quad\left[f,h\right]=\layer{0}{2}f$$
\end{lem}
\begin{proof}
Considering the formula mentioned in \Rref{elem-com-formula} we get:
\begin{eqnarray*}
\left[e,f\right] & = & \left[e_{1,2},e_{2,1}\right]=e_{1,1}+\minus e_{2,2}=h\\
\left[e,h\right] & = & \left[e_{1,2},e_{1,1}+\minus e_{2,2}\right]=\minus e_{1,2}+\minus e_{1,2}=\layer 0{-2}e\\
\left[f,h\right] & = & \left[e_{2,1},e_{1,1}+\minus e_{2,2}\right]=e_{2,1}+\minus\minus e_{2,1}=\layer 02f
\end{eqnarray*}
\end{proof}
Let $L$ be the Lie semialgebra constructed in \Eref{exa:sl2-type}. Note that $A_1$ is similar to $L$; however, $A_1$ has no d,s-base, whereas $L$ is free.\\

Unlike the definition of $A_{n}$, the Lie semialgebras $B_{n}$, $C_{n}$ and $D_{n}$ are all defined using involutions. In this part, we follow \cite{Rowen2006}.

\begin{defn}
Let $R$ be a semifield with a negation map, and let $A_{1},A_{2}$ be nonassociative semialgebras. An \textbf{antihomomorphism} is a function $\varphi:A_{1}\rightarrow A_{2}$ that satisfies:
\begin{enumerate}
\item $\forall x,y\in A_{1}:\varphi\left(x+y\right)=\varphi\left(x\right)+\varphi\left(y\right)$.
\item $\forall\alpha\in R,\,\forall x\in A_{1}:\varphi\left(\alpha x\right)=\alpha\varphi\left(x\right)$.
\item $\forall x,y\in A_{1}:\varphi\left(xy\right)=\varphi\left(y\right)\varphi\left(x\right)$.
\end{enumerate}
\end{defn}

\begin{defn}
Let $R$ be a semifield with a negation map, and let $A$ be a nonassociative semialgebra. An \textbf{involution} is an antihomomorphism $\varphi:A\rightarrow A$ for which
$$\forall x\in A:\varphi\left(\varphi\left(x\right)\right)=x$$
In other words, $\varphi$ is its own inverse. We also denote involutions as $*:A\rightarrow A$, $x\mapsto x^{*}$.
\end{defn}

\begin{example}
Let $R$ be a semifield with a negation map. Then the transpose on $M_{n\times n}\left( R\right)$ is an involution.
\end{example}

\begin{defn}
Let $R$ be a semifield with a negation map, and let $A$ be a nonassociative semialgebra with an involution $*$. An element $x\in A$ is called \textbf{symmetric}, if
$$x^{*}=x$$
$x$ is called \textbf{skew-symmetric}, if
$$x^{*}=\minus x$$
\end{defn}

\begin{lem}\label{lem:skew-syms-lie-alg}
Let $A$ be a nonassociative semialgebra over $R$. Denote by $\left[\cdot,\cdot\right]$ the negated commutator on $A$, and let $*$ be an involution on $A$. Then the set of all skew-symmetric elements in $A$,
$$\tilde{A}=\left\{ x\in A\mid x^{*}=\minus x\right\}$$
is a Lie subalgebra of $A$.
\end{lem}
\begin{proof}
$*$ is an involution, and in particular an antihomomorphism, so $\tilde{A}$ is a submodule of $A$.\\

Assume $x,y\in A$. Then
$$\left[x,y\right]^{*}=\left(xy+\minus yx\right)^{*}=y^{*}x^{*}+\minus x^{*}y^{*}=\left[y^{*},x^{*}\right]$$
If $x,y\in\tilde{A}$, then
$$\left[x,y\right]^{*}=\left[y^{*},x^{*}\right]=\left[\minus y,\minus x\right]=\left[y,x\right]=\minus\left[x,y\right]$$
so $\left[x,y\right]\in\tilde{A}$, as needed.
\end{proof}

Now we can define $B_{n}$, $C_{n}$ and $D_{n}$ by defining their involutions.

\begin{defn}
Let $R$ be a semifield with a negation map. We define an involution $*$ on $\gl\left(k, R\right)$ by the transpose ($A^{*}=A^{t}$). As stated in \Lref{lem:skew-syms-lie-alg}, we get a Lie subalgebra consisting all of the skew-symmetric elements.
\begin{enumerate}
\item When $k=2n+1$ is odd, this Lie semialgebra is called $B_{n}$, the \textbf{negated odd-dimensional orthogonal algebra}.
\item When $k=2n$ is even, it is called $D_{n}$, the \textbf{negated even-dimensional orthogonal algebra}.
\end{enumerate}
\end{defn}

\begin{rem}
$B_{n}$ has the following s-base:
$$\left\{ e_{i,j}+\minus e_{j,i}\mid1\leq i<j\leq2n+1\right\}\cup\left\{\zero e_{i,i}\mid 1\leq i\leq 2n+1\right\}$$
whereas $D_{n}$ has the following s-base:
$$\left\{ e_{i,j}+\minus e_{j,i}\mid1\leq i<j\leq2n\right\}\cup\left\{\zero e_{i,i}\mid 1\leq i\leq 2n\right\}$$
\end{rem}

\begin{defn}
Let $R$ be a semifield with a negation map. We define an involution $*$ on $\gl\left(2n, R\right)$ by
$$\begin{pmatrix}A & B\\
C & D
\end{pmatrix}^{*}=\begin{pmatrix}D^{t} & \minus B^{t}\\
\minus C^{t} & A^{t}
\end{pmatrix}$$
By taking the skew-symmetric elements, we get the Lie subalgebra $C_{n}$, the \textbf{negated symplectic algebra}.
\end{defn}

\begin{lem}
$C_{n}$ has the following base:
$$\left\{ e_{i,j}+\minus e_{n+j,n+i}\mid1\leq i,j\leq n\right\} \cup\left\{ e_{i,n+j}+e_{j,n+i}\mid1\leq i,j\leq n\right\} \cup\left\{ e_{n+i,j}+e_{n+j,i}\mid1\leq i,j\leq n\right\}$$
\end{lem}
\begin{proof}
Denote the above set by $B$ , and assume $\begin{pmatrix}A_{1} & A_{2}\\
A_{3} & A_{4}
\end{pmatrix}\in C_{n}$. Then
$$\begin{pmatrix}A_{1} & A_{2}\\
A_{3} & A_{4}
\end{pmatrix}^{*}=\begin{pmatrix}A_{4}^{t} & \minus A_{2}^{t}\\
\minus A_{3}^{t} & A_{1}^{t}
\end{pmatrix}=\minus\begin{pmatrix}A_{1} & A_{2}\\
A_{3} & A_{4}
\end{pmatrix}$$
we get the following conditions:
$$A_{4}^{t}=\minus A_{1};\quad A_{2}^{t}=A_{2};\quad A_{3}^{t}=A_{3}$$
If $\left(A_{k}\right)_{i,j}=a_{k,i,j}$, then
$$\begin{pmatrix}A_{1} & A_{2}\\
A_{3} & A_{4}
\end{pmatrix}=\sum_{i,j=1}^{n}a_{1,i,j}\left(e_{i,j}+\minus e_{n+j,n+i}\right)+\sum_{i,j=1}^{n}a_{2,i,j}\left(e_{i,n+j}+e_{j,n+i}\right)+\sum_{i,j=1}^{n}a_{3,i,j}\left(e_{n+i,j}+e_{n+j,i}\right)$$
and this combination is unique, implying $B$ is a base of $C_{n}$.
\end{proof}

\section{Solvable and Nilpotent Symmetrized Lie Algebras}
\subsection{Basic Definitions}
\subsubsection{Solvability}
\begin{defn}
Let $R$ be a semifield with a negation map, and let $L$ be a Lie semialgebra with a negation map over~$R$. The \textbf{derived series} of $L$ is the series defined by $L^{\left(0\right)}=L$, $L^{\left(1\right)}=L'=\left[L,L\right]$, and in general: $L^{\left(n\right)}=\left[L^{\left(n-1\right)},L^{\left(n-1\right)}\right]$.
\end{defn}

\begin{defn}
Let $R$ be a semifield with a negation map, and let $L$ be a Lie semialgebra with a negation map over~$R$. We say that $L$ is \textbf{solvable}, if
$$\exists n\in\mathbb{N}:L^{\left(n\right)}\subseteq\zeroset{L}$$
\end{defn}

\begin{example}
All abelian Lie semialgebras with a negation map are solvable.
\end{example}

\begin{lem}
Any Lie semialgebra with a negation map $L$, for which $L'\subseteq\zeroset{L}$, is abelian.
\end{lem}
\begin{proof}
The condition is equivalent to $\forall x,y\in L:\left[x,y\right]\in\zeroset{L}$, which shows that $L$ is abelian.
\end{proof}

\begin{lem}\label{lem:subalg-homim-solv}
Let $R$ be a semifield with a negation map, and let $L$ be a Lie semialgebra with a negation map over $R$. If $L$ is solvable, then so are all of its subalgebras and homomorphic images.
\end{lem}
\begin{proof}
If $K$ is a subalgebra of $L$, then it is easy to see by induction that $\forall n\in\mathbb{N}:K^{\left(n\right)}\subseteq L^{\left(n\right)}$.\\

Also, if $\varphi:L\rightarrow L_1$ is an $R$-epimorphism, then
$$\forall n\in\mathbb{N}:\left(L_1\right)^{\left(n\right)}=\varphi\left(L^{\left(n\right)}\right)$$
Therefore, if $L$ is solvable, there is some $n\in\mathbb{N}$ such that $L^{\left(n\right)}=\zeroset{L}$; by the above formula, $\left(L_1\right)^{\left(n\right)}\subseteq\zeroset{L_1}$.
\end{proof}

\begin{lem}\label{lem:solvable-series}
Let $R$ be a semifield with a negation map, and let $L$ be a Lie semialgebra with a negation map over $R$. $L$ is solvable if and only if there exists a chain of subalgebras $L_{k}\ideal L_{k-1}\ideal\cdots\ideal L_{0}=L$, such that $L_{k}\subseteq\zeroset{L}$ and $L_{i}'\subseteq L_{i+1}$.
\end{lem}
\begin{proof}
It is easy to check that in such a chain, $L^{\left(i\right)}\subseteq L_{i}$;
so, if such a chain exists, $L^{\left(k\right)}\subseteq\zeroset{L}$,
implying $L$ is solvable.\\

If $L$ is solvable with $L^{\left(k\right)}\subseteq\zeroset{L}$,
than the chain $L^{\left(k\right)}\ideal\cdots\ideal L^{\left(0\right)}=L$
works.
\end{proof}

\begin{lem}\label{lem:intersect-sol-is-sol}
If $I,J\ideal L$ are two ideals of $L$, and if $I$ is solvable, then $I\cap J$ is also solvable.
\end{lem}
\begin{proof}
$I\cap J\ideal I$, so this is a direct corollary of \Lref{lem:subalg-homim-solv}.
\end{proof}

\begin{lem}\label{lem:sum-sol-is-sol}
If $I,J\ideal L$ are two solvable ideals of $L$, then $I+J$ is also solvable.
\end{lem}
\begin{proof}
We shall prove by induction that
$$\forall k\in\mathbb{N}\cup\left\{ 0\right\} :\left(I+J\right)^{\left(k\right)}\subseteq I^{\left(k\right)}+J^{\left(k\right)}+I\cap J$$

If $k=0$ the assertion is clear. Assume it is true for $k$, i.e.
$$\left(I+J\right)^{\left(k\right)}\subseteq I^{\left(k\right)}+J^{\left(k\right)}+I\cap J$$
We get
\begin{eqnarray*}
\left(I+J\right)^{\left(k+1\right)} & = & \left[\left(I+J\right)^{\left(k\right)},\left(I+J\right)^{\left(k\right)}\right]\subseteq\left[I^{\left(k\right)}+J^{\left(k\right)}+I\cap J,I^{\left(k\right)}+J^{\left(k\right)}+I\cap J\right]=\\
 & = & \left[I^{\left(k\right)},I^{\left(k\right)}\right]+\left[I^{\left(k\right)},J^{\left(k\right)}+I\cap J\right]+\left[J^{\left(k\right)},J^{\left(k\right)}\right]+\left[J^{\left(k\right)},I^{\left(k\right)}+I\cap J\right]+\\
 & + & \left[I\cap J,I^{\left(k\right)}+J^{\left(k\right)}+I\cap J\right]\subseteq I^{\left(k+1\right)}+J^{\left(k+1\right)}+I\cap J
\end{eqnarray*}
Note that $\left[I^{\left(k\right)},J^{\left(k\right)}+I\cap J\right],\left[J^{\left(k\right)},I^{\left(k\right)}+I\cap J\right],\left[I\cap J,I^{\left(k\right)}+J^{\left(k\right)}+I\cap J\right]\subseteq I\cap J$.\\

$I$ and $J$ are solvable, so $\exists m,n\in\mathbb{N}:I^{\left(m\right)},J^{\left(n\right)}\subseteq\zeroset{L}$. Taking $k=\max\left\{ m,n\right\} $, we get
$$\left(I+J\right)^{\left(k\right)}\subseteq\zeroset{L}+I\cap J$$
We also know that $\left(I\cap J\right)^{\left(k\right)}\subseteq\zeroset{L}$, and thus,
$$\left(I+J\right)^{\left(2k\right)}\subseteq\left(\zeroset{L}\right)^{\left(k\right)}+\left(I\cap J\right)^{\left(k\right)}+\zeroset{L}\cap I\cap J\subseteq\zeroset{L}$$
implying $I+J$ is solvable.
\end{proof}

\subsubsection{Nilpotency}
\begin{defn}
Let $R$ be a semifield with a negation map, and let $L$ be a Lie semialgebra with a negation map over~$R$. The \textbf{descending central series}, or the \textbf{lower central series}, is the series defined by $L^{0}=L$, $L'=L^{1}=\left[L,L\right]$, and in general: $L^{n}=\left[L,L^{n-1}\right]$.
\end{defn}

\begin{defn}
Let $R$ be a semifield with a negation map, and let $L$ be a Lie semialgebra with a negation map over~$R$. $L$ is called \textbf{nilpotent}, if
$$\exists n\in\mathbb{N}:L^{n}\subseteq\zeroset{L}$$
\end{defn}

\begin{example}
All abelian Lie semialgebras with a negation map are nilpotent.
\end{example}

\begin{lem}
$\forall n\in\mathbb{N}\cup\left\{0\right\} :L^{\left(n\right)}\subseteq L^{n}$.
\end{lem}
\begin{proof}
By induction: for $n=0$, it is clear, because $L^{\left(0\right)}=L=L^{0}$.\\

Assume the statement is true for $n\in\mathbb{N}\cup\left\{ 0\right\} $, and let $x\in L^{\left(n+1\right)}$. Then there exist $y_{1},y_{2}\in L^{\left(n\right)}$ such that
$$x=\left[y_{1},y_{2}\right]$$
So, $y_{1}\in L$, and by the induction hypothesis, $y_{2}\in L^{n}$. Thus,
$$x=\left[y_{1},y_{2}\right]\in\left[L,L^{n}\right]=L^{n+1}$$
and we get $L^{\left(n\right)}\subseteq L^{n}$.
\end{proof}

\begin{cor}
Any nilpotent Lie semialgebra with a negation map is solvable.
\end{cor}

\begin{lem}
Let $R$ be a semifield with a negation map, and let $L$ be a nilpotent Lie semialgebra with a negation map over $R$.
\begin{enumerate}
\item All of the subalgebras and homomorphic images of $L$ are nilpotent.
\item If $L\neq\zeroset{L}$, then $\zeroset{L}\neq Z\left(L\right)$ (in fact, $\zeroset{L}\subsetneq Z\left(L\right)$).
\end{enumerate}
\end{lem}
\begin{proof}
~
\begin{enumerate}
\item If $K$ is a subalgebra of $L$, then it is easy to see by induction that $\forall n\in\mathbb{N}:K^{n}\subseteq L^{n}$. Also, if $\varphi:L\rightarrow L'$ is an $R$-epimorphism, then $\forall n\in\mathbb{N}:\left(L'\right)^{n}=\varphi\left(L^{n}\right)$.
\item Assume $n$ is the maximal index for which $L^{n}\nsubseteq\zeroset{L}$. Then $\left[L,L^{n}\right]\subseteq\zeroset{L}$, so $L^{n}\subseteq Z\left(L\right)$. $\exists x_{0}\in L^{n}\backslash\zeroset{L}$, so $x_{0}\in Z\left(L\right)\backslash\zeroset{L}$,
as required.
\end{enumerate}
\end{proof}

\begin{lem}
If $L'$ is nilpotent, then $L$ is solvable.
\end{lem}
\begin{proof}
We shall prove that
$$\forall n\in\mathbb{N}\cup\left\{ 0\right\} :L^{\left(n+1\right)}\subseteq\left(L'\right)^{n}$$
We use induction on $n$, the case $n=0$ being trivial. We get
$$L^{\left(n+1\right)}=\left[L^{\left(n\right)},L^{\left(n\right)}\right]\subseteq\left[\left(L'\right)^{n-1},\left(L'\right)^{n-1}\right]\subseteq\left[L',\left(L'\right)^{n-1}\right]=\left(L'\right)^{n}$$
Thus, if $L'$ is nilpotent, for some $n$ we have $\left(L'\right)^{n}\subseteq\zeroset{\left(L'\right)}\subseteq\zeroset{L}$, implying $L^{\left(n+1\right)}\subseteq\zeroset{L}$, so $L$ is solvable.
\end{proof}

Similarly to \Lref{lem:intersect-sol-is-sol} and \Lref{lem:sum-sol-is-sol}, we have the following lemmas:

\begin{lem}
If $I,J\ideal L$ are two ideals of $L$, and if $I$ is nilpotent, then $I\cap J$ is also nilpotent.
\end{lem}

\begin{lem}
If $I,J\ideal L$ are two nilpotent ideals of $L$, then $I+J$ is also nilpotent.
\end{lem}

\subsubsection{Solvability and Nilpotency Modulo Ideals}

In the semiring, the quotient algebra structure fails, and the replacement is congruences. However, we would define equivalent definitions to solvability and nilpotency of quotient algebras.
\begin{defn}
Let $R$ be a semifield with a negation map, let $L$ be a Lie semialgebra with a negation map over $R$, and let $I\ideal L$. We say that $L$ is \textbf{solvable modulo $I$}, if
$$\exists n\in\mathbb{N}:L^{\left(n\right)}\subseteq I$$
\end{defn}

\begin{rem}
Solvability modulo $\zeroset{L}$ is equivalent to solvability.
\end{rem}

\begin{lem}
If $L$ is solvable modulo $I$, and if $I$ is solvable (as a Lie semialgebra with a negation map), then~$L$ is also solvable.
\end{lem}
\begin{proof}
$L$ is solvable modulo $I$, so there exists $n\in\mathbb{N}$ for which $L^{\left(n\right)}\subseteq I$. But $I$ is also solvable, so there exists $m\in\mathbb{N}$ for which $I^{\left(m\right)}\subseteq\zeroset I\subseteq\zeroset{L}$. To conclude,
$$L^{\left(m+n\right)}=\left(L^{\left(n\right)}\right)^{\left(m\right)}\subseteq I^{\left(m\right)}\subseteq\zeroset{L}$$
implying $L$ is solvable.
\end{proof}

\begin{defn}
Let $R$ be a semifield with a negation map, let $L$ be a Lie semialgebra with a negation map over $R$, and let $I\ideal L$. We say that $L$ is \textbf{nilpotent modulo $I$}, if
$$\exists n\in\mathbb{N}:L^{n}\subseteq I$$
\end{defn}

\begin{rem}
Nilpotency modulo $\zeroset{L}$ is equivalent to nilpotency.
\end{rem}

\begin{defn}
Let $R$ be a semifield with a negation map, let $L$ be a Lie semialgebra with a negation map over $R$, and let $I\ideal L$. We say that $L$ is \textbf{abelian modulo $I$}, if $\left[L,L\right]\subseteq I$.
\end{defn}

\begin{rem}
\Lref{lem:solvable-series} can be rephrased as follows: $L$ is solvable if and only if there exists a chain of subalgebras $L_{k}\ideal L_{k-1}\ideal\cdots\ideal L_{0}=L$ such that each $L_{i+1}$ is abelian modulo $L_{i}$ and $L_{k}$ is abelian.
\end{rem}

\subsubsection{The Symmetrized Radical and Semisimplicity}
\begin{defn}
Let $R$ be a semifield with a negation map, let $L$ be a Lie semialgebra with a negation map over $R$. The (\textbf{negated}) \textbf{radical} of $L$ is
$$\Rad L=\sum_{I\ideal L\textnormal{ is solvable}}I$$
That is, the negated radical is the sum of all solvable ideals of $L$.
\end{defn}

\begin{rem}
In the classical theory, the radical of a finitely generated Lie algebra
 is its unique maximal solvable ideal. However, in our theory, there is no guarantee that there would be one maximal solvable ideal, since being finitely generated as a module with a negation map usually does not mean being Noetherian (and thus, there may be infinite ascending chains of solvable ideals).
\end{rem}

\begin{defn}
Let $R$ be a semifield with a negation map, let $L$ be a Lie semialgebra over $R$. We say that $L$ is \textbf{locally solvable}, if any finitely generated subalgebra of $L$ is solvable.
\end{defn}

\begin{lem}
$\Rad{L}$ is locally solvable.
\end{lem}
\begin{proof}
Let $L_1\subseteq\Rad{L}$ be a finitely generated subalgebra of $L$. Write
$$L_1=\Span\left\{x_1,\dots,x_n\right\}$$
By the definition of $\Rad{L}$, for each $i$ there are solvable ideals $I_{i,1},\dots,I_{i,m_i}$ such that
$$x_i\in\sum_{j=1}^{m_i} I_{i,j}$$
Therefore,
$$L_1\subseteq\sum_{i=1}^n\sum_{j=1}^{m_i}I_{i,j}$$
Since the sum is a finite sum of solvable ideals, by \Lref{lem:sum-sol-is-sol}, it is also solvable. Therefore, by \Lref{lem:subalg-homim-solv}, also $L_1$ is solvable, as required.
\end{proof}

\begin{defn}
Let $R$ be a semifield with a negation map, let $L$ be a Lie semialgebra over $R$. If~$\Rad L=\zeroset{L}$, then $L$ is called \textbf{semisimple}.
\end{defn}

\begin{lem}\label{lem:cond-for-semisimple}
$L$ is semisimple if and only if any abelian ideal of $L$ is contained in $\zeroset{L}$.
\end{lem}
\begin{proof}
Let $L$ be semisimple, and let $I$ be an abelian ideal of $L$. Then $I$ is solvable, implying $I\subseteq\Rad L=\zeroset{L}$.\\

Now suppose that any abelian ideal of $L$ is contained in $\zeroset{L}$. If $L$ is not semisimple, it has a solvable ideal $I\nsubseteq\zeroset{L}$. Take a maximal $n$ such that $I^{\left(n\right)}\nsubseteq\zeroset{L}$; then $I^{\left(n\right)}$ is an abelian ideal, a contradiction.
\end{proof}

\subsection{Lifts of Lie Semialgebras with a Negation Map}

As a continuation of \sSref{sec:Lifts-Semiring-Mod}, we give a definition of a lift of a Lie semialgebra with a negation map.

\begin{defn}
Let $R$ be a semiring with a negation map with a lift $\left(\widehat{R},\widehat{\varphi}\right)$, and let $L$ be a Lie semialgebra over $R$. A \textbf{lift} of $L$ is a Lie algebra over $\widehat{R}$, $\widehat{L}$, and a map $\widehat{\psi}:\widehat{L}\to L$, such that the following properties hold:
\begin{enumerate}
  \item $\left(\widehat{L},\widehat{\psi}\right)$ is a lift of $L$ as modules.
  \item $\forall x_1,x_2\in \widehat{L}:\widehat{\psi}\left(\left[x_1,x_2\right]\right)\asymmdash\left[\widehat{\psi}\left(x_1\right),\widehat{\psi}\left(x_2\right)\right]$.
\end{enumerate}
\end{defn}

We also have two parallel lemmas to \Lref{lem:lift-submodule}:
\begin{lem}\label{lem:lift-subalgebra}
Let $\widehat{K}\subseteq\widehat{L}$ be a subalgebra. Then $\overline{\widehat{\psi}\left(\widehat{K}\right)}$ is a subalgebra of $L$.
\end{lem}
\begin{proof}
By \Lref{lem:lift-submodule}, $\overline{\widehat{\psi}\left(\widehat{K}\right)}$ is a submodule of $L$. Therefore, we need to check that it is closed with respect to the negated Lie bracket.\\

Let $x,y\in\overline{\widehat{\psi}\left(\widehat{K}\right)}$. Take $\widehat{x},\widehat{y}\in\widehat{K}$ such that $x\symmdash\widehat{\psi}\left(\widehat{x}\right)$ and $y\symmdash\widehat{\psi}\left(\widehat{y}\right)$. Then
$$\left[x,y\right]\symmdash\left[\widehat{\psi}\left(\widehat{x}\right),\widehat{\psi}\left(\widehat{y}\right)\right]\symmdash\widehat{\psi}\left(\left[\widehat{x},\widehat{y}\right]\right)$$
Since $\left[\widehat{x},\widehat{y}\right]\in\widehat{K}$, we are done.
\end{proof}

\begin{lem}\label{lem:lift-ideal}
Let $\widehat{I}\ideal\widehat{L}$ be an ideal. Then $\overline{\widehat{\psi}\left(\widehat{I}\right)}$ is an ideal of $L$.
\end{lem}
\begin{proof}
By \Lref{lem:lift-subalgebra}, $\overline{\widehat{\psi}\left(\widehat{I}\right)}$ is a subalgebra of $L$.\\

Let $x\in\overline{\widehat{\psi}\left(\widehat{K}\right)}$ and $y\in L$. Take $\widehat{x}\in\widehat{I}$ such that $x\symmdash\widehat{\psi}\left(\widehat{x}\right)$, and take $\widehat{y}\in\widehat{L}$ such that $y\symmdash\widehat{\psi}\left(\widehat{y}\right)$. Then
$$\left[x,y\right]\symmdash\left[\widehat{\psi}\left(\widehat{x}\right),\widehat{\psi}\left(\widehat{y}\right)\right]\symmdash\widehat{\psi}\left(\left[\widehat{x},\widehat{y}\right]\right)$$
Since $\left[\widehat{x},\widehat{y}\right]\in\widehat{I}$, we are done.
\end{proof}

Although in general there is no obvious way to find a lift of a given Lie semialgebra (even if there is a lift as modules), there is some consolation:
\begin{thm}
Let $R$ be a semiring with a negation map with a lift $\left(\widehat{R},\widehat{\varphi}\right)$. Consider the function
$$\widehat{\psi}:\gl\left(n,\widehat{R}\right)\to\gl\left(n,R\right)$$
which applies $\widehat{\varphi}$ on each entry of a given matrix. Then $\left(\gl\left(n,\widehat{R}\right),\widehat{\psi}\right)$ is a lift of $\gl\left(n,R\right)$ as Lie semialgebras.
\end{thm}
\begin{proof}
Since this is already a lift of modules, we will prove that given $\widehat{A},\widehat{B}\in\gl\left(n,\widehat{R}\right)$,
$$\widehat{\psi}\left(\widehat{A}\right)\widehat{\psi}\left(\widehat{B}\right)\symmdash\widehat{\psi}\left(\widehat{A}\widehat{B}\right)$$
Indeed,
$$\left(\widehat{\psi}\left(\widehat{A}\right)\widehat{\psi}\left(\widehat{B}\right)\right)_{i,j}=\sum_{k=1}^{n}\widehat{\varphi}\left(\widehat{a}_{i,k}\right)\widehat{\varphi}\left(\widehat{b}_{k,j}\right)\symmdash\widehat{\varphi}\left(\sum_{k=1}^{n}\widehat{a}_{i,k}\widehat{b}_{k,j}\right)=\left(\widehat{\psi}\left(\widehat{A}\widehat{B}\right)\right)_{i,j}$$

Now the assertion follows, since
\begin{eqnarray*}
  \left[\widehat{\psi}\left(\widehat{A}\right),\widehat{\psi}\left(\widehat{B}\right)\right] &=& \widehat{\psi}\left(\widehat{A}\right)\widehat{\psi}\left(\widehat{B}\right)+\minus\widehat{\psi}\left(\widehat{B}\right)\widehat{\psi}\left(\widehat{A}\right)\symmdash\widehat{\psi}\left(\widehat{A}\widehat{B}\right)+\minus\widehat{\psi}\left(\widehat{B}\widehat{A}\right)\symmdash \\
   &\symmdash& \widehat{\psi}\left(\widehat{A}\widehat{B}+\minus\widehat{B}\widehat{A}\right)=\widehat{\psi}\left(\left[\widehat{A},\widehat{B}\right]\right)
\end{eqnarray*}
\end{proof}

\subsubsection{Solvability, Nilpotency and Semisimplicity of the Lift}

Throughout this subsection, $\left(\widehat{R},\widehat{\varphi}\right)$ will be a lift of a semifield with a negation map $R$, $L$ will be a Lie semialgebra over $R$, and $\left(\widehat{L},\widehat{\psi}\right)$ will be a lift of $L$.

\begin{lem}\label{lem:comm-of-ELTrops}
Let $\widehat{L}_{1},\widehat{L}_{2}$ be subalgebras of $\widehat{L}$. Then
$$\left[\overline{\widehat{\psi}\left(\widehat{L}_{1}\right)},\overline{\widehat{\psi}\left(\widehat{L}_{2}\right)}\right]\subseteq\overline{\widehat{\psi}\left(\left[\widehat{L}_{1},\widehat{L}_{2}\right]\right)}$$
\end{lem}
\begin{proof}
Let $x\in\overline{\widehat{\psi}\left(\widehat{L}_{1}\right)}$ and $y\in\overline{\widehat{\psi}\left(\widehat{L}_{2}\right)}$. We will prove that $\left[x,y\right]\in\overline{\widehat{\psi}\left(\left[\widehat{L}_{1},\widehat{L}_{2}\right]\right)}$. Take $\widehat{x}\in\widehat{L}_1$ and $\widehat{y}\in\widehat{L}_2$ such that $x\symmdash\widehat{\psi}\left(\widehat{x}\right)$ and $y\symmdash\widehat{\psi}\left(\widehat{y}\right)$. Then
$$\left[x,y\right]\symmdash\left[\widehat{\psi}\left(\widehat{x}\right),\widehat{\psi}\left(\widehat{y}\right)\right]\symmdash\widehat{\psi}\left(\left[\widehat{x},\widehat{y}\right]\right)$$
Since $\left[\widehat{x},\widehat{y}\right]\in\left[\widehat{L}_{1},\widehat{L}_{2}\right]$, we are finished.
\end{proof}

\begin{cor}\label{cor:ELTrop-L^n-and-L^(n)}
$$\forall n\in\N\cup\left\{0\right\}:L^{n}\subseteq\overline{\widehat{\psi}\left[\widehat{L}^{n}\right]}$$
and
$$\forall n\in\N\cup\left\{0\right\}:L^{\left(n\right)}\subseteq\overline{\widehat{\psi}\left[\widehat{L}^{\left(n\right)}\right]}$$
\end{cor}
\begin{proof}
We will prove the first assertion; the second is proven similarly.\\

We use induction on $n$, the case $n=0$ being trivial. If the assertion holds for $n\in\N\cup\left\{0\right\}$, then, by \Lref{lem:comm-of-ELTrops} and the induction hypothesis,
$$L^{n+1}=\left[L,L^n\right]\subseteq\left[\overline{\widehat{\psi}\left(\widehat{L}\right)}, \overline{\widehat{\psi}\left(\widehat{L}^n\right)}\right]\subseteq\overline{\widehat{\psi}\left(\left[\widehat{L},\widehat{L}^n\right]\right)}=\overline{\widehat{\psi}\left(L^{n+1}\right)}$$
as required.
\end{proof}

\begin{cor}\label{cor:ELTrop-of-solnil-is-solnil}
If $\widehat{L}$ is nilpotent (resp. solvable), then $L$ is nilpotent (resp. solvable).
\end{cor}
\begin{proof}
We will prove the assertion for nilpotency; the assertion for solvability is proved similarly. Since $\widehat{L}$ is nilpotent, there is $n$ such that $\widehat{L}^n=0$. By \Cref{cor:ELTrop-L^n-and-L^(n)},
$$L^n\subseteq\overline{\widehat{\psi}\left(\widehat{L}^n\right)}=\overline{\widehat{\psi}\left(0\right)}=\zeroset{L}$$
and thus $L$ is nilpotent.
\end{proof}

\begin{cor}\label{cor:ELTrop-of-Rad}
$$\overline{\widehat{\psi}\left(\Rad\left(\widehat{L}\right)\right)}\subseteq\Rad\left(L\right)$$
\end{cor}
\begin{proof}
$\Rad\left(\widehat{L}\right)\ideal\widehat{L}$, and thus, by \Lref{lem:lift-ideal}, $\overline{\widehat{\psi}\left(\Rad\left(\widehat{L}\right)\right)}\ideal L$. In addition, by \Cref{cor:ELTrop-of-solnil-is-solnil}, $\overline{\widehat{\psi}\left(\Rad\left(\widehat{L}\right)\right)}$ is solvable. Therefore, the assertion follows.
\end{proof}

\begin{cor}\label{cor:lift-of-semisimple-is-semisimple}
If $L$ is semisimple, then $\widehat{L}$ is semisimple.
\end{cor}
\begin{proof}
$L$ is semisimple, hence $\Rad\left(L\right)\subseteq\zeroset{L}$. By \Cref{cor:ELTrop-of-Rad},
$$\overline{\widehat{\psi}\left(\Rad\left(\widehat{L}\right)\right)}\subseteq\Rad\left(L\right)\subseteq\zeroset{L}$$
Hence $\Rad\left(\widehat{L}\right)=0$, as required.
\end{proof}

\subsection{Cartan's Criterion for Lie Algebras Over ELT Algebras}

For this subsection, we work only with ELT algebras. Specifically, we are interested in the following type of ELT algebra:
\begin{defn}
Let $\R=\ELT{\L}{\F}$ be an ELT algebra. We say that $\R$ is a \textbf{divisible ELT field}, if $\L$ is a field, and $\F$ is a divisible group.
\end{defn}

In our version of Cartan's criterion, we will use the essential trace defined in \cite{Blachar2016b}. Let us recall its definition:
\begin{defn}
Let $\R$ be a divisible ELT field, and let $A\in M_n\left(\overline{\R}\right)$. Let
$$p_A\left(\lambda\right)=\det\left(\lambda I+\layer{0}{-1} A\right)= \lambda^n+\sum_{i=1}^n\alpha_i\lambda^{n-i}$$
be its ELT characteristic polynomial. We define
\begin{eqnarray*}
  L\left(A\right) &=& \left\{\ell\ge 1\mid \forall k\ge n:\frac{\t\left(\alpha_\ell\right)}{\ell}\ge \frac{\t\left(\alpha_k\right)}{k}\right\}\\
  \mu\left(A\right) &=& \min L\left(A\right)
\end{eqnarray*}
In words, $\alpha_\mu\lambda^{n-\mu}$ is the first monomial after $\lambda^n$ which is not inessential in $p_A\left(\lambda\right)$.\\

The \textbf{essential trace} of $A$ is
$$\etr\left(A\right)=\left\{\begin{matrix}\tr\left(A\right),&\layer{0}{-1}\tr\left(A\right)\textrm{ is essential in }p_A\left(\lambda\right)\\\layer{\left(\frac{\t\left(\alpha_\mu\right)}{\mu}\right)}{0},&\textrm{Otherwise}\end{matrix}\right.$$
\end{defn}

Recall:

\begin{cor*}[{\cite[Corollary 2.22]{Blachar2016b}}]
If $A$ is ELT nilpotent, then $s\left(\etr\left(A\right)\right)=0$.
\end{cor*}

Returning to Lie semialgebras, we specialize our symmetrized Lie semialgebra to be free; therefore, given a homomorphism $\varphi:L\to L$, its trace and essential trace are well-defined.

\begin{defn}
Let $\R$ be a divisible ELT field, and let $L$ be a free symmetrized Lie semialgebra over $\R$. The \textbf{Killing form} of $L$ is
$$\kil{x}{y}=\tr\left(\ad_x\circ \ad_y\right)$$
The \textbf{essential Killing form} of $L$ is
$$\esskil{x}{y}=\etr\left(\ad_x\circ \ad_y\right)$$
\end{defn}

\begin{defn}
Let $\R$ be a divisible ELT field, let $L$ be a free symmetrized Lie semialgebra over~$\R$, and let $\kappa$ be the Killing form of $L$. We say that $\kappa$ is \textbf{non-degenerate}, if
$$\Rad\kappa=\left\{x\in L\middle|\forall y\in L:s\left(\kil{x}{y}\right)=0\right\}\subseteq\zeroset{L}$$
We say that $\kappa^{\mathrm{es}}$ is \textbf{non-degenerate}, if
$$\SRad\kappa=\left\{x\in L\middle|\forall y\in L:s\left(\esskil{x}{y}\right)=0\right\}\subseteq\zeroset{L}$$
\end{defn}

\begin{rem}
Since $\Rad\kappa\subseteq\SRad\kappa$ (by \cite[Lemma 2.11]{Blachar2016b}), if $\kappa^{\mathrm{es}}$ is non-degenerate, then $\kappa$ is non-degenerate.
\end{rem}

\begin{thm}[Cartan's criterion]\label{thm:Cartan's-crit}
If $L$ has a non-degenerate essential Killing form, then $L$ is semisimple.
\end{thm}
\begin{proof}
By \Lref{lem:cond-for-semisimple}, it is enough to prove that any abelian ideal $I$ of $L$ is of layer zero.\\

Let $I$ be an abelian ideal, and let $x\in I$. Take an arbitrary $y\in L$. Since $\ad_x\circ\ad_y$ maps $L$ to $I$, $\left(\ad_x\circ\ad_y\right)^2$ maps $L$ to $\left[I,I\right]$, which is of layer zero. So $\ad_x\circ\ad_y$ is ELT nilpotent. Thus, by \cite[Corollary 2.22]{Blachar2016b}, $s\left(\etr\left(\ad_x\circ\ad_y\right)\right)=0$.\\

This argument implies $x\in\SRad\kappa$; so $I\subseteq\SRad\kappa$. But $\kappa$ is strongly non-degenerate, and thus $I\subseteq\zeroset{L}$, as required. So $L$ is semisimple.
\end{proof}

\section{PBW Theorem}

In this subsection we will construct the universal enveloping algebra of an arbitrary Lie semialgebra with a negation map. Then, we will give a counterexample to a naive version of PBW theorem.\\

The tensor product is a well-known process in general categories (for example, \cite{Banagl2013}, \cite{Grothendieck1955}, \cite{Katsov1978} and \cite{Takahashi1982}). Regarding tensor product of modules over semirings, there are two (nonisomorphic) notions of tensor product of modules over semirings in the literature, both denoted by $\otimes_R$. For our purposes, we use the tensor product discussed in \cite{Katsov1997}, \cite{Katsov2005} and \cite{Banagl2013}. The other notion of tensor product is discussed in \cite{Takahashi1982} and \cite{Golan1999}. Their construction satisfies a different universal property, and the outcome is a cancellative monoid.

\subsection{Tensor Power and Tensor Algebra of Modules over a Semirings}

\begin{defn}
Let $R$ be a commutative semiring, and let $M$ be an $R$-module. The \textbf{$k$-th tensor product} of $M$ is defined as
$$M^{\otimes k}=\underbrace{M\otimes\dots\otimes M}_{k\textnormal{ times}}$$
If $k=0$, we define $M^{\otimes 0}=R$.
\end{defn}

We note that this definition is well-defined, because the tensor product is associative.

\begin{defn}
Let $R$ be a commutative semiring, and let $M$ be an $R$-module. The \textbf{tensor algebra} of $M$, denoted $T\left(M\right)$, is
$$T\left(M\right)=\bigoplus_{k=0}^\infty M^{\otimes k}$$
\end{defn}

\begin{rem}
The natural isomorphism
$$M^{\otimes k}\otimes M^{\otimes\ell}\cong M^{\otimes\left(k+\ell\right)}$$
yields a multiplication
$$M^{\otimes k}\times M^{\otimes\ell}\to M^{\otimes\left(k+\ell\right)}$$
given by $\left(a,b\right)\mapsto a\otimes b$. Endowed with this multiplication, $T\left(M\right)$ is an $R$-semialgebra.
\end{rem}

\begin{lem}\label{lem:univ-prop-of-tensor-algebra}
Let $A$ be an $R$-semialgebra, and assume that $\varphi:M\to A$ is a homomorphism of $R$-modules. Then there exists a unique homomorphism of $R$-semialgebras $\hat{\varphi}:T\left(M\right)\to A$, such that $\left.\hat{\varphi}\right|_M=\varphi$.
\end{lem}
\begin{proof}
Uniqueness forces that
$$\hat{\varphi}\left(x_1\otimes\dots\otimes x_k\right)=\hat{\varphi}\left(x_1\right)\otimes\dots\otimes\hat{\varphi}\left(x_k\right)= \varphi\left(x_1\right)\otimes\dots\otimes\varphi\left(x_k\right)$$
This defines a homomorphism of $R$-semialgebras.
\end{proof}

\subsection{The Universal Enveloping Algebra of a Lie Algebra with a Negation Map}

We now return to our motivation, where we work with a Lie semialgebra $L$ over a semifield with a negation map $R$.
\begin{defn}
Let $R$ be a semifield with a negation map, and let $L$ be a Lie semialgebra over~$R$. We define a congruence $\sim$ on $T\left(L\right)$ as the congruence generated by
$$\forall x,y\in L:\left[x,y\right]\sim x\otimes y+\minus y\otimes x$$
The \textbf{universal enveloping algebra} of $L$ is
$$U\left(L\right)=\quo{T\left(L\right)}{\sim}$$

If $\rho:T\left(L\right)\to U\left(L\right)$ is the canonical map, we have the \textbf{PBW homomorphism} of $L$, which we denote $\varphi_L:L\to U\left(L\right)$, defined by $\varphi_L=\left.\rho\right|_L$.
\end{defn}

However, the PBW homomorphism of $L$ need not be injective, as the next subsection demonstrates.
\subsection{Counterexample to the Naive PBW Theorem}
We now present a Lie semialgebra with a negation map, which cannot be embedded into a Lie semialgebra with a negation map obtained from an associative algebra together with the negated commutator operator. In particular, $\varphi_L$ will not be injective.

\begin{example}\label{exa:disproof-PBW}

The example will use the ELT structre. For convenience, we assume that our underlying ELT field is $\R=\ELT{\mathbb{R}}{\mathbb{C}}$. Recall that in the ELT theory, the relation $\symmdash$ is denoted $\vDash$.\\

We take a free $\R$-module $L$ with base $B=\left\{x_1,x_2\right\}$, and define:
\begin{eqnarray*}
  \left[x_1,x_1\right] &=& \layer{1}{0}x_1+\layer{1}{0}x_2 \\
  \left[x_2,x_2\right] &=& \layer{1}{0}x_1+\layer{1}{0}x_2 \\
  \left[x_1,x_2\right] &=& x_1+x_2\\
  \left[x_2,x_1\right] &=& \minus \left[x_1,x_2\right]
\end{eqnarray*}
By \Lref{lem:2-dim-free-ELT-Lie-alg}, $L$ equipped with $\left[\cdot,\cdot\right]$ is a Lie semialgebra with a negation map over $\R$.\\

Consider
$$y_1=\left[x_1,\left[x_1,x_2\right]\right]+\left[x_1,\left[x_2,x_1\right]\right]$$
and
$$y_2=\left[x_2,\left[x_1,x_1\right]\right].$$

From Jacobi's identity, we know that $y_1+y_2\in\zeroset{L}$. Let us expand them, using the definition of the Lie bracket:
$$y_1=\left[x_1,\left[x_1,x_2\right]\right]+\left[x_1,\left[x_2,x_1\right]\right]=\left[x_1,\layer{0}{0}\left[x_1,x_2\right]\right]= \left[x_1,\layer{0}{0}x_1+\layer{0}{0}x_2\right]=$$
$$=\layer{0}{0}\left[x_1,x_1\right]+\layer{0}{0}\left[x_1,x_2\right]= \layer{1}{0}x_1+\layer{1}{0}x_2+\layer{0}{0}x_1+\layer{0}{0}x_2=\layer{1}{0}x_1+\layer{1}{0}x_2$$
whereas
$$y_2=\left[x_2,\left[x_1,x_1\right]\right]=\left[x_2,\layer{1}{0}x_1+\layer{1}{0}x_2\right]=\layer{1}{0}\left[x_2,x_1\right]+\layer{1}{0}\left[x_2,x_2\right]=$$
$$=\layer{1}{0}\left(\layer{0}{-1}x_1+\layer{0}{-1}x_2\right)+\layer{1}{0}\left(\layer{1}{0}x_1+\layer{1}{0}x_2\right)=\layer{2}{0}x_1+\layer{2}{0}x_2$$
Therefore, $y_2\vDash y_1$, but $y_1\neq y_2$.\\

Now, assume that there is an ELT associative algebra $A$, which is a Lie semialgebra with a negation map together with the ELT commutator, such that there exists an $\R$-monomorphism $\varphi:L\to A$ of Lie semialgebra with a negation maps. Denote $a_1=\varphi\left(x_1\right)$, $a_2=\varphi\left(x_2\right)$. We will show that $\varphi\left(y_1\right)=\varphi\left(y_2\right)$, which will show that $\varphi$ is not injective.\\

We note that
$$b_1=\varphi\left(y_1\right)=\left[a_1,\left[a_1,a_2\right]\right]+\left[a_1,\left[a_2,a_1\right]\right]$$
and
$$b_2=\varphi\left(y_2\right)=\left[a_2,\left[a_1,a_1\right]\right]$$
Since any associative algebra with the symmetrized commutator satisfies the strong Jacobi's identity (\Eref{exa:comm-is-Lie}), $b_1\vDash b_2$. However, since $y_2\vDash y_1$, also $b_1=\varphi\left(y_2\right)\vDash\varphi\left(y_1\right)=b_2$.\\

By \Cref{cor:R-is-idempotent}, since $\zeroset{\R}$ is idempotent, $\vDash$ is a partial order on $L$. In particular, $\varphi\left(y_1\right)=b_1=b_2=\varphi\left(y_2\right)$. But $\varphi$ is a monomorphism, implying $y_1=y_2$, a contradiction.
\end{example}
\bibliographystyle{plain}
\bibliography{references}

\end{document}